\documentclass[3p,number,sort&compress]{elsarticle}
\usepackage{hyperref,url}
\usepackage{amssymb}
\usepackage{amsmath}

\newcommand{\cB}{\mathcal{B}}
\newcommand{\cE}{\mathcal{E}}
\newcommand{\cJ}{\mathcal{J}}
\newcommand{\cK}{\mathcal{K}}
\newcommand{\cQ}{\mathcal{Q}}

\newcommand{\C}{\mathbb{C}}
\newcommand{\N}{\mathbb{N}}
\newcommand{\R}{\mathbb{R}}

\newcommand{\eps}{\varepsilon}

\newcommand{\norm}[1]{\left\Vert#1\right\Vert}
\newcommand{\abs}[1]{\left\vert#1\right\vert}

\newtheorem{theorem}{Theorem}[section]
\newtheorem{lemma}[theorem]{Lemma}
\newtheorem{corollary}[theorem]{Corollary}
\newproof{proof}{Proof}
\numberwithin{equation}{section}

\begin{document}
\title{On singular integral operators
with semi-almost periodic coefficients
on variable Lebesgue spaces}

\author[AK]{Alexei Yu. Karlovich\corref{Karl}}
\ead{oyk@fct.unl.pt}

\author[IS]{Ilya M. Spitkovsky}
\ead{ilya@math.wm.edu}

\cortext[Karl]{Corresponding author}

\address[AK]{%%
Departamento de Matem\'atica,
Faculdade de Ci\^encias e Tecnologia,
Universidade Nova de Lisboa,
Quinta da Torre,
2829--516 Caparica,
Portugal}

\address[IS]{%%
Department of Mathematics,
College of William \& Mary,
Williamsburg, VA, 23187-8795,
U.S.A.}

\begin{abstract}
Let $a$ be a semi-almost periodic matrix function with the almost
periodic representatives $a_l$ and $a_r$ at $-\infty$ and
$+\infty$, respectively. Suppose $p:\R\to(1,\infty)$ is a slowly
oscillating exponent such that the Cauchy singular integral
operator $S$ is bounded on the variable Lebesgue space
$L^{p(\cdot)}(\R)$. We prove that if the operator $aP+Q$ with
$P=(I+S)/2$ and $Q=(I-S)/2$ is Fredholm on the variable Lebesgue
space $L_N^{p(\cdot)}(\R)$, then the operators $a_lP+Q$ and
$a_rP+Q$ are invertible on standard Lebesgue spaces
$L_N^{q_l}(\R)$ and $L_N^{q_r}(\R)$ with some exponents $q_l$ and
$q_r$ lying in the segments between the lower and the upper limits
of $p$ at $-\infty$ and $+\infty$, respectively.
\end{abstract}


\maketitle

%%%%%%%%%%%%%%%%%%%%%%%%%%%%%%%%%%%%%%%%%%%%%%%%%%%%%%%%%%%%%%%%%%%%%%%%%%%
\section{Introduction}
Given a Banach space $X$, we denote by $X_N$ the Banach space of
all columns of height $N$ with components in $X$; the norm in
$X_N$ is defined by
\[
\|(x_1,\dots,x_N)^\mathrm{T}\|_{X_N}=\left(\sum_{\alpha=1}^N\|x_\alpha\|_X^2\right)^{1/2}.
\]
Given a subalgebra $B$ of $L^\infty(\R)$, we denote by $B_{N\times
N}$ the algebra of all $N\times N$ matrices with entries in $B$;
we equip $B_{N\times N}$ with the norm
\[
\|a\|_{B_{N\times N}}=\|(a_{\alpha\beta})_{\alpha,\beta=1}^N\|_{B_{N\times N}}=
\left(\sum_{\alpha,\beta=1}^N\|a_{\alpha\beta}\|_B^2\right)^{1/2}.
\]
Let $\cB(X)$ denote the Banach algebra of all
bounded linear operators on $X$ and let $\cK(X)$ denote the ideal of all
compact operators on $X$. As usual, $A^*$ denotes the adjoint operator
of $A\in\cB(X)$. An operator $A\in\cB(X)$ is said to be Fredholm on $X$
if its image $\operatorname{Im}A$ is closed
in $X$ and
\[
\dim\operatorname{Ker}A<\infty,
\quad
\dim(X/\overline{\operatorname{Im}A})<\infty.
\]

We denote by $C(\overline{\R})$ the set of all complex-valued
continuous functions $c$ on $\R$ which have finite limits $c(-\infty)$
and $c(+\infty)$ at $-\infty$ and $+\infty$. Let $C(\dot{\R})$ be the set
of all functions $c\in C(\overline{\R})$ such that $c(-\infty)=c(+\infty)$.
An almost-periodic polynomial is a function of the form
%%%
%\begin{equation}\label{a}
\[
a(x)=\sum_{j=1}^m a_j e^{i\lambda_j x} \quad
(x\in\R)\quad\mbox{with}\quad a_j\in\C,\quad\lambda_j\in\R.
\]
%\end{equation}
%%%
The set of all almost-periodic polynomials will be denoted by
$AP^0$. The algebra $AP$ of the continuous almost-periodic
functions is defined as the closure of $AP^0$ in $L^\infty(\R)$;
its closure with respect to a stronger {\em Wiener} norm
$\norm{a}_W:=\sum\abs{a_j}$ is the algebra $APW$. Note that $APW$
is dense in $AP$. Finally, the algebra $SAP$ of the
semi-almost-periodic functions is the smallest closed subalgebra
of $L^\infty(\R)$ containing $C(\overline{\R})\cup AP$. The
algebra $SAP$ was introduced by Sarason \cite{Sarason77}, who also
showed that every $a\in SAP_{N\times N}$ can be written in the
form
\[
a=(1-u)a_l+ua_r+a_0,
\]
where $u\in C(\overline{\R})$ is any fixed function such that
$0\le u\le 1$, $u(-\infty)=0$, $u(+\infty)=1$, $a_l$ and $a_r$
belong to $AP_{N\times N}$, and $a_0$ is in $[C_0]_{N\times N}$,
the set of all continuous matrix functions vanishing at $-\infty$
and $+\infty$. Moreover, $a_l$ and $a_r$ are uniquely determined
by $a$ and the maps $a\mapsto a_l$ and $a\mapsto a_r$ are
$C^*$-algebra homomorphisms of $SAP_{N\times N}$ onto $AP_{N\times
N}$. The matrix functions $a_l$ and $a_r$ are referred to as the
almost-periodic representatives of $a$ at $-\infty$ and $+\infty$,
respectively (for $N=1$, see \cite[Theorem~1.21]{BKS02}; for $N>1$,
the proof is the same).

For a continuous function $f:\R\to\C$ and $J\subset\R$, let
\[
\operatorname{osc}(f,J):=\sup_{t,\tau\in J}|f(t)-f(\tau)|.
\]
Following \cite{Power80}, we denote by $SO$ the class of slowly
oscillating functions given by
\[
SO:=\left\{f\in C(\R) \colon\lim_{x\to+\infty}\operatorname{osc}(f,[-2x,-x]\cup[x,2x])=0\right\}\cap
L^\infty(\R).
\]
Clearly, $SO$ is a unital $C^*$-subalgebra of $L^\infty(\R)$ that contains $C(\dot{\R})$.

Let $p\colon\R\to[1,\infty]$ be a measurable a.e. finite function. By
$L^{p(\cdot)}(\R)$ we denote the set of all complex-valued functions
$f$ on $\R$ such that
\[
I_{p(\cdot)}(f/\lambda):=\int_\R |f(x)/\lambda|^{p(x)} dx <\infty
\]
for some $\lambda>0$. This set becomes a Banach space when
equipped with the norm
\[
\|f\|_{p(\cdot)}:=\inf\big\{\lambda>0: I_{p(\cdot)}(f/\lambda)\le 1\big\}.
\]
It is easy to see that if $p$ is constant, then $L^{p(\cdot)}(\R)$ is nothing but
the standard Lebesgue space $L^p(\R)$. The space $L^{p(\cdot)}(\R)$
is referred to as a \textit{variable Lebesgue space}.
We will always suppose that
%%%
\begin{equation}\label{eq:exponents}
1<p_-:=\operatornamewithlimits{ess\,inf}_{x\in\R}p(x),
\quad
\operatornamewithlimits{ess\,sup}_{x\in\R}p(x)=:p_+<\infty.
\end{equation}
%%%
Under these conditions, the space $L^{p(\cdot)}(\R)$ is separable and reflexive,
its dual space is isomorphic to the space $L^{p'(\cdot)}(\R)$, where
\[
1/p(x)+1/p'(x)=1 \quad(x\in\R)
\]
(see, e.g., \cite{KR91}).

The Cauchy singular integral operator $S$ is defined for $f\in L^1_{\rm loc}(\R)$
by
\[
(Sf)(x):=\frac{1}{\pi i}\int_\R\frac{f(\tau)}{\tau-x}d\tau
\quad(x\in\R),
\]
where the integral is understood in the principal value sense. Assume that
$S$ generates a bounded operator on $L^{p(\cdot)}(\R)$ and put
\[
P:=(I+S)/2,
\quad
Q:=(I-S)/2.
\]
The operators $S,P$, and $Q$ are defined on $L_N^{p(\cdot)}(\R)$ elementwise.
If $a\in L^\infty_{N\times N}(\R)$, then the operator $aI$ of multiplication
by $a$ is bounded on $L_N^{p(\cdot)}(\R)$. We will say that
the operator $aP+Q$ with $a\in L_{N\times N}^\infty(\R)$
is a singular integral operator with the coefficient $a$.

A Fredholm criterion for Banach algebras of singular integral
operators with piecewise continuous coefficients on variable
Lebesgue spaces $L^{p(\cdot)}(\Gamma,w)$ over Carleson Jordan
curves with weights having finite sets of singularities were
obtained in \cite{K09,K10,K11} (see also the references therein).
The approach of these works is based on further developments of
the methods of the monograph \cite{BK97} based on localization
techniques, Wiener-Hopf factorization and heavy use of results and
methods from the theory of submultiplicative functions. An
alternative approach to Fredholm theory of singular integral
operators with piecewise continuous and slowly oscillating
coefficients is based on the method of limit operators and Mellin
pseudodifferential operators techniques (we refer to \cite{R96},
\cite[Section~10.6]{BK97}, \cite{BKR00} in the case of standard
Lebesgue spaces and to \cite{RS08,RS11} in the case of weighted
variable Lebesgue spaces). The second approach allows one to treat
the case of composed curves, but still not arbitrary composed
Carleson curves.

Notice that in all mentioned works coefficients are piecewise
continuous or slowly oscillating; and the variable exponent $p$ is
continuous and has a finite limit at infinity in the case of
unbounded curves. The aim of the present paper is to make the
first step beyond these hypotheses: we are going to study singular
integral operators $aP+Q$ with $a\in SAP_{N\times N}$ on variable
exponent spaces with the exponent $p$ which may not have a limit
at infinity.

Let $\cE$ denote the class of exponents $p:\R\to[1,\infty]$
satisfying \eqref{eq:exponents}, continuous on $\R$, and such that
the Cauchy singular integral operator is bounded on $L^{p(\cdot)}(\R)$.
First, we observe that this class contains interesting exponents.
%%%%%%%%%%%%%%%%%%%%%%%%%%%%%%%%%%%%%%%%%%%%%%%%%%%%%%%%%%%%%%%%%%%%%%%%%%%
\begin{lemma}\label{le:nontriviality}
There exists an exponent $p\in\cE$ such that $p\in SO\setminus C(\dot{\R})$.
\end{lemma}
%%%%%%%%%%%%%%%%%%%%%%%%%%%%%%%%%%%%%%%%%%%%%%%%%%%%%%%%%%%%%%%%%%%%%%%%%%%
Lerner \cite{L05} constructed an example of a variable exponent $p_L\notin C(\dot{\R})$
such that the Hardy-Littlewood maximal operator $M$ is bounded on $L^{p_L(\cdot)}(\R)$.
It is known that the boundedness of the Hardy-Littlewood
maximal operator implies the boundedness of the Cauchy singular integral
operator \cite{D05,DR03,KL05}. Thus $p_L\in\cE$. It turns out that $p_L\in SO$,
which gives the proof of Lemma~\ref{le:nontriviality}. All details of the proof
of Lemma~\ref{le:nontriviality} are contained in Section~\ref{sec:E-nontriviality}.

Our main result is the following.
%%%%%%%%%%%%%%%%%%%%%%%%%%%%%%%%%%%%%%%%%%%%%%%%%%%%%%%%%%%%%%%%%%%%%%%%%%%
\begin{theorem}[Main result]
\label{th:main}
Let $a\in SAP_{N\times N}$ and $p\in\cE\cap SO$.
If the operator $aP+Q$ is Fredholm on
the variable Lebesgue space $L_N^{p(\cdot)}(\R)$, then
\begin{enumerate}[(a)]
\item
there is an exponent $q_r$ lying in the segment
\[
\left[\liminf\limits_{x\to+\infty}p(x),\limsup\limits_{x\to+\infty}p(x)\right]
\]
such that $a_rP+Q$ is invertible on the standard
Lebesgue space $L_N^{q_r}(\R)$.

\item
there is an exponent $q_l$ lying in the segment
\[
\left[\liminf\limits_{x\to-\infty}p(x),\limsup\limits_{x\to-\infty}p(x)\right]
\]
such that $a_lP+Q$ is invertible on the standard
Lebesgue space $L_N^{q_l}(\R)$.
\end{enumerate}
\end{theorem}
%%%%%%%%%%%%%%%%%%%%%%%%%%%%%%%%%%%%%%%%%%%%%%%%%%%%%%%%%%%%%%%%%%%%%%%%%%%
For standard Lebesgue spaces this result boils down to the
statement that Fredholmness of $aP+Q$ with $a\in SAP_{N\times N}$
on $L_N^p(\R)$ implies the invertibility of $a_rP+Q$, $a_lP+Q$ on
the same space $L_N^p(\R)$, and in this form was established in
\cite{K93} (see also its proof in \cite[Chap.~18]{BKS02}).

Note also that if $b\in APW_{N\times N}$, then the operator $bP+Q$
is invertible on all standard Lebesgue spaces $L^p_N(\R)$,
$1<p<\infty$, as soon as it is invertible on at least one of them
(see \cite[Section~18.1]{BKS02}). It is not known at the moment
whether this property persists for all $b\in AP_{N\times N}$. In
particular, we do not know whether in the setting of
Theorem~\ref{th:main} the operators $a_lP+Q$ and $a_rP+Q$ are
invertible on $L^{q_l}(\R)$ and $L^{q_r}(\R)$ for all $q_l$ and
$q_r$ in the segments between the lower and the upper limits of
$p$ at $-\infty$ and $+\infty$, respectively.

The proofs in \cite{K93,BKS02} are based on the method of limit
operators. The outline of this method is as follows. Let $h\in\R$
and $V_h$ be the translation operator given by
\[
(V_hf)(x):=f(x-h)\quad (x\in\R).
\]
It is well known that this operator is an isometry on every standard Lebesgue
space. Moreover, it commutes with the Cauchy singular integral operator $S$.
The method of limit operators consists in the study of the strong limits of $V_{-h_k}AV_{h_k}$
as $k\to\infty$ for a given operator $A$ and a given sequence $\{h_k\}_{k=1}^\infty$
tending to $+\infty$ or to $-\infty$. Typically, these strong limits
(if they exist) are simpler than the original operator $A$, but still
keep much information about $A$. For instance $V_{-h_k}KV_{h_k}$ tends
strongly to the zero operator for every compact operator $K$ on the standard
Lebesgue space and $V_{-h_k}(aP+Q)V_{h_k}$ tends strongly to $a_lP+Q$
for $h_k\to-\infty$ and to $a_rP+Q$ for $h_k\to+\infty$.
For a detailed discussion of the method of limit operators,
we refer to the monograph by Rabinovich, Roch, and Silbermann \cite{RRS04}.

On variable Lebesgue spaces $L^{p(\cdot)}(\R)$ the translation
operator $V_h$ is, in general, unbounded. So the method of the
proof of Theorem~\ref{th:main} presented in
\cite[Section~18.4]{BKS02} should be adjusted accordingly. To this
end, we combine ideas from \cite[Section~18.4]{BKS02} and
\cite{RS08} (see also \cite{RS11}). A key lemma concerns the
behavior of the sequence $\|V_{h_k}w_k\|_{p(\cdot)}$, where the
functions $w_k$ are nice (continuous and decaying faster than
$|x|$ as $|x|\to+\infty$): if $w_k$ converges to $w$ and $p(h_k)$
converges to $q\in(1,\infty)$, then $\|V_{h_k}w_k\|_{p(\cdot)}$
converges to the norm of $w$ on the standard Lebesgue space
$L^q(\R)$. This fact was proved by Rabinovich and Samko
\cite[Proposition~6.3]{RS08} for exponents having finite limits at
infinity; we relax this hypothesis and assume only that $p\in SO$.

The paper is organized as follows.
Section~\ref{sec:E-nontriviality} contains the proof of
Lemma~\ref{le:nontriviality}. In Section~\ref{sec:auxiliary} we
collect auxiliary material on Fredholmness, injectivity and
surjectivity moduli and their relation with invertibility, some
fundamental properties of variable Lebesgue spaces. Further, we
prove that $P$ and $Q$ are projections on variable Lebesgue spaces
and calculate the adjoint operator of $aP+bQ$ with $a,b\in
L_{N\times N}^\infty(\R)$. We prove that the sequence
$\Psi_n=K\chi_{\R\setminus[-n,n]}I$ converges uniformly whenever
$K$ is compact on $L_N^{p(\cdot)}(\R)$. We finish this section
with a property of slowly oscillating functions and an implicit
sequence lemma. Both statements play an important role in the proof
of the key lemma given in Section~\ref{sec:key}.

The final Section~\ref{sec:main-proof} is devoted to the proof of
Theorem~\ref{th:main}. Let us briefly outline its main steps.
First we approximate the operator $A=aP+Q$ by the operators
$A_j=a_jP+Q$ where $a_j$ has the same form as $a$, but with
polynomial almost-periodic representatives $a_l^{(j)}$ and
$a_l^{(j)}$ at $-\infty$ and $+\infty$, respectively. Since the
norm of $K\Psi_n$ is small whenever $n$ is large, from the
Fredholmness of $A$ we arrive at an a priori estimate
%%%
\begin{equation}\label{eq:intro-1}
\|\Psi_n f\|_{L_N^{p(\cdot)}(\R)}\le\mathrm{const}\|A_j\Psi_n f\|_{L_N^{p(\cdot)}(\R)}
\quad\mbox{for}\quad
f\in L_N^{p(\cdot)}(\R)
\end{equation}
%%%
and large fixed $j,n$. By the corollary of Kronecker's theorem there
exists a sequence $h_m\to+\infty$ such that
%%%
\begin{equation}\label{eq:intro-2}
\|a_r^{(j)}(\cdot+h_m)-a_r^{(j)}(\cdot)\|_{L_{N\times N}^\infty(\R)}\to 0
\quad\mbox{as}\quad m\to\infty.
\end{equation}
%%%
If $\varphi$ is smooth and compactly supported, $\varphi\in
[C_c^\infty(\R)]_N$, then $\Psi_nV_{h_m}\varphi=V_{h_m}\varphi$ for
large $m$. Hence \eqref{eq:intro-1} implies that
%%%
\begin{equation}\label{eq:intro-3}
\|V_{h_m}\varphi\|_{L_N^{p(\cdot)}(\R)}\le\mathrm{const}\|V_{h_m}(V_{-h_m}A_jV_{h_m}\varphi)\|_{L_N^{p(\cdot)}(\R)}
\quad\mbox{for}\quad\varphi\in [C_c^\infty(\R)]_N.
\end{equation}
%%%
Since the sequence $\{p(h_m)\}$ is bounded, we can extract its
subsequence $\{p(h_{m_k})\}$ that converges to a certain number
$q_r$. Taking into account \eqref{eq:intro-2}, we show that the
sequence $w_k=V_{-h_{m_k}}A_jV_{h_{m_k}}\varphi$ and the function
$w:=(a_r^{(j)}P+Q)\varphi$ satisfy the hypotheses of the key lemma.
Passing to the limit in
\eqref{eq:intro-3} along the subsequence $\{h_{m_k}\}$ as
$k\to\infty$, and then replacing $a_r^{(j)}$ by $a_r$, we arrive at
%%%
\begin{equation}\label{eq:intro-4}
\|\varphi\|_{L_N^{q_r}(\R)}\le\mathrm{const}\|(a_r P+Q)\varphi\|_{L_N^{q_r}(\R)}
\quad\mbox{for}\quad\varphi\in [C_c^\infty(\R)]_N.
\end{equation}
%%%
Applying duality arguments, we also obtain an a priori estimate
for the adjoint operator:
%%%
\begin{equation}\label{eq:intro-5}
\|\varphi\|_{L_N^{q_r'}(\R)}\le\mathrm{const}\|(a_r P+Q)^*\varphi\|_{L_N^{q_r'}(\R)}
\quad\mbox{for}\quad\varphi\in [C_c^\infty(\R)]_N
\end{equation}
%%%
where $q_r'=q_r/(q_r-1)$. Since $C_c^\infty(\R)$ is dense
both in $L^{p(\cdot)}(\R)$ and in its dual space
$L^{p'(\cdot)}(\R)$ whenever \eqref{eq:exponents} is fulfilled, from
\eqref{eq:intro-4}--\eqref{eq:intro-5} it follows that the operator
$a_rP+Q$ is invertible on $L_N^{q_r}(\R)$.
%%%%%%%%%%%%%%%%%%%%%%%%%%%%%%%%%%%%%%%%%%%%%%%%%%%%%%%%%%%%%%%%%%%%%%%%%%%
\section{Nontriviality of the class $\cE$}\label{sec:E-nontriviality}
\subsection{The Hardy-Littlewood maximal operator and the Cauchy singular integral operator}
Given $f\in  L^1_{\rm loc}(\R)$, the Hardy-Littlewood maximal
operator $M$ is defined by
\[
(Mf)(x) := \sup_{Q\ni x}\frac{1}{|Q|}\int_Q |f(y)|dy
\]
where the supremum is taken over all intervals $Q\subset\R$
containing $x$. From \cite[Theorem~4.8]{DR03} (see also
\cite[Theorem~2.7]{KL05}) and \cite[Theorem~8.1]{D05} one can
extract the following.
%%%%%%%%%%%%%%%%%%%%%%%%%%%%%%%%%%%%%%%%%%%%%%%%%%%%%%%%%%%%%%%%%%%%%%%%
\begin{theorem}\label{th:M-S}
Let $p:\R\to[1,\infty]$ be a measurable function satisfying \eqref{eq:exponents}.
If the Hardy-Littlewood maximal operator $M$ is bounded on $L^{p(\cdot)}(\R)$,
then the Cauchy singular integral operator $S$ is bounded on $L^{p(\cdot)}(\R)$.
\end{theorem}
%%%%%%%%%%%%%%%%%%%%%%%%%%%%%%%%%%%%%%%%%%%%%%%%%%%%%%%%%%%%%%%%%%%%%%%
Note that in the majority of papers dealing with the boundedness
of the Hardy-Littlewood maximal operator it is supposed that the
exponent has a finite limit at infinity (see, e.g.,
\cite{CFN03,CFN04,D04,H09,KS08} and the references therein). We
refer also to \cite{L10,N08}, where this condition was weakened
and to the recent monograph \cite{DHHR11} for the detailed
treatment of these questions.
%%%%%%%%%%%%%%%%%%%%%%%%%%%%%%%%%%%%%%%%%%%%%%%%%%%%%%%%%%%%%%%%%%%%%%%
\subsection{Lerner's example}
One interesting class of variable exponents such that $M$ is bounded
on $L^{p(\cdot)}(\R)$ was considered by Lerner \cite{L05}. Among other
things he proved the following.
%%%%%%%%%%%%%%%%%%%%%%%%%%%%%%%%%%%%%%%%%%%%%%%%%%%%%%%%%%%%%%%%%%%%%%%
\begin{theorem}[Lerner]
\label{th:Lerner}
There exists an $\alpha>2$ such that the Hardy-Littlewood maximal operator
$M$ is bounded on the variable Lebesgue space $L^{p_L(\cdot)}(\R)$ with
\[
p_L(x):=\alpha+\sin\big(\log(\log|x|)\chi_{\{x\in\R:|x|\ge e\}}(x)\big)
\quad(x\in\R).
\]
\end{theorem}
%%%%%%%%%%%%%%%%%%%%%%%%%%%%%%%%%%%%%%%%%%%%%%%%%%%%%%%%%%%%%%%%%%%%%%%%%%%
\begin{lemma}\label{le:Lerner-SO}
The exponent $p_L$ satisfies \eqref{eq:exponents} and belongs to $SO\setminus C(\dot{\R})$.
\end{lemma}
%%%%%%%%%%%%%%%%%%%%%%%%%%%%%%%%%%%%%%%%%%%%%%%%%%%%%%%%%%%%%%%%%%%%%%%%%%%
\begin{proof}
It is clear that $p_L\in C(\R)$ and $p_L$ is even. Moreover,
\[
\lim_{x\to+\infty}x\frac{dp_L(x)}{dx}=\lim_{x\to+\infty}\frac{\cos(\log(\log x))}{\log x}=0.
\]
Then (see, e.g., \cite[p.~154--155 and p.~158]{BBK04}) $p_L\in SO$.
Obviously,
\[
\liminf_{x\to+\infty}p_L(x)=\inf_{x\in\R}p_L(x)=\alpha-1>1,
\quad
\limsup_{x\to+\infty}p_L(x)=\sup_{x\in\R}p_L(x)=\alpha+1<\infty.
\]
Thus $p_L$ satisfies \eqref{eq:exponents} and  $p_L\notin C(\dot{\R})$.
\qed
\end{proof}
%%%%%%%%%%%%%%%%%%%%%%%%%%%%%%%%%%%%%%%%%%%%%%%%%%%%%%%%%%%%%%%%%%%%%%%%%%%
Lemma~\ref{le:nontriviality} follows from
Theorems~\ref{th:M-S}--\ref{th:Lerner} and
Lemma~\ref{le:Lerner-SO}.
%%%%%%%%%%%%%%%%%%%%%%%%%%%%%%%%%%%%%%%%%%%%%%%%%%%%%%%%%%%%%%%%%%%%%%%%%%%
\section{Auxiliary results}
\label{sec:auxiliary}
%%%%%%%%%%%%%%%%%%%%%%%%%%%%%%%%%%%%%%%%%%%%%%%%%%%%%%%%%%%%%%%%%%%%%%%%%%%
\subsection{Fredholmness}
Recall the following well known fact, which follows from Atkinson's
theorem (see, e.g., \cite[Chap.~4, Theorem~6.1]{GK92}).
%%%%%%%%%%%%%%%%%%%%%%%%%%%%%%%%%%%%%%%%%%%%%%%%%%%%%%%%%%%%%%%%%%%%%%%%%%%
\begin{lemma}\label{le:Atkinson}
Let $X$ be a Banach space and $A,B\in\cB(X)$. If $A$ is Fredholm on $X$ and
$B$ is invertible on $X$, then $AB$ and $BA$ are Fredholm on $X$.
\end{lemma}
%%%%%%%%%%%%%%%%%%%%%%%%%%%%%%%%%%%%%%%%%%%%%%%%%%%%%%%%%%%%%%%%%%%%%%%%%%%
The next statement is about Fredholmness of adjoints.
%%%%%%%%%%%%%%%%%%%%%%%%%%%%%%%%%%%%%%%%%%%%%%%%%%%%%%%%%%%%%%%%%%%%%%%%%%%
\begin{theorem}[{see, e.g., \cite[Section~4.15]{GK92}}]
\label{th:Fredholm-duality}
Let $X$ be a Banach space and $A\in\cB(X)$. Then $A$ is Fredholm on $X$
if and only if its adjoint $A^*$ is Fredholm on the dual space $X^*$.
\end{theorem}
%%%%%%%%%%%%%%%%%%%%%%%%%%%%%%%%%%%%%%%%%%%%%%%%%%%%%%%%%%%%%%%%%%%%%%%%%%%
Let $A\in\cB(X)$. An operator $R\in\cB(X)$ is said to be a left (resp. right)
regularizer of $A$ if $RA-I\in\cK(X)$ (resp. $AR-I\in\cK(X)$). If $R$ is
a left and right regularizer of $A$, then we say that $R$ is a two-sided
regularizer of $A$.
%%%%%%%%%%%%%%%%%%%%%%%%%%%%%%%%%%%%%%%%%%%%%%%%%%%%%%%%%%%%%%%%%%%%%%%%%%%
\begin{theorem}[{see, e.g., \cite[Chap.~4, Theorem~7.1]{GK92}}]
\label{th:regularization}
Let $X$ be a Banach space.
An operator $A\in\cB(X)$ is Fredholm on $X$ if and only if there exists
a two-sided regularizer of $A$.
\end{theorem}
%%%%%%%%%%%%%%%%%%%%%%%%%%%%%%%%%%%%%%%%%%%%%%%%%%%%%%%%%%%%%%%%%%%%%%%%%%%
\subsection{Injection and surjection moduli}
Let $A\in\cB(X)$. Following \cite[Sections~B.3.1 and B.3.4]{P80}, consider its injection
modulus
\[
\cJ(A;X):=\sup\big\{c\ge 0: \ \|Af\|_X\ge c \|f\|_X\text{ for all } f\in X\big\}
\]
and its surjection modulus
\[
\cQ(A;X):=\sup\big\{c\ge 0: \ cB_X\subset AB_X\big\}
\]
where $B_X$ is the closed unit ball of $X$. Sometimes these
characteristics are also called lower norms of $A$ (see, e.g.,
\cite[Section~1.3]{Ku99}). Fundamental properties of the injection and
surjection moduli are collected in the following statements.
%%%%%%%%%%%%%%%%%%%%%%%%%%%%%%%%%%%%%%%%%%%%%%%%%%%%%%%%%%%%%%%%%%%%%%%%%%%
\begin{lemma}[{see, e.g., \cite[Section~B.3.8]{P80}}]
\label{le:injection-surjection-duality}
If $A\in\cB(X)$, then
\[
\cJ(A;X)=\cQ(A^*;X^*),
\quad
\cQ(A;X)=\cJ(A^*;X^*).
\]
\end{lemma}
%%%%%%%%%%%%%%%%%%%%%%%%%%%%%%%%%%%%%%%%%%%%%%%%%%%%%%%%%%%%%%%%%%%%%%%%%%%
\begin{lemma}[{see, e.g., \cite[Proposition~1.3.7]{Ku99}}]
\label{le:moduli-supermult}
If $A,B\in\cB(X)$, then
\[
\cJ(A;X)\cdot\cJ(B;X)\le\cJ(AB;X),
\quad
\cQ(A;X)\cdot\cQ(B;X)\le\cQ(AB;X).
\]
\end{lemma}
%%%%%%%%%%%%%%%%%%%%%%%%%%%%%%%%%%%%%%%%%%%%%%%%%%%%%%%%%%%%%%%%%%%%%%%%%%%
\begin{theorem}[{see, e.g., \cite[Theorem~1.3.2]{Ku99}}]
\label{th:Kurbatov}
An operator $A\in\cB(X)$ is invertible if and only if
\[
\cJ(A;X)>0,\quad\cQ(A;X)>0.
\]
If $A$ is invertible, then
\[
\cJ(A;X)=\cQ(A;X)=\frac{1}{\|A^{-1}\|_{\cB(X)}}.
\]
\end{theorem}
%%%%%%%%%%%%%%%%%%%%%%%%%%%%%%%%%%%%%%%%%%%%%%%%%%%%%%%%%%%%%%%%%%%%%%%%%%%
\subsection{Some fundamental properties of variable Lebesgue spaces}
Let $C_c^\infty(\R)$ be the set of all infinitely differentiable functions
with compact support. The following results were proved in
\cite[Theorems~2.4, 2.6, and 2.11]{KR91}.
%%%%%%%%%%%%%%%%%%%%%%%%%%%%%%%%%%%%%%%%%%%%%%%%%%%%%%%%%%%%%%%%%%%%%%%%%%%
\begin{theorem}\label{th:KR}
Let $p:\R\to[1,\infty]$ be a measurable function satisfying
\eqref{eq:exponents} and $f_n\in\ L^{p(\cdot)}(\R)$. Then
\begin{enumerate}[(a)]
\item
the set $C_c^\infty(\R)$ is dense in $L^{p(\cdot)}(\R)$;

\item
$\lim\limits_{n\to\infty}I_{p(\cdot)}(f_n)=0$ if and only if $\lim\limits_{n\to\infty}\|f_n\|_{p(\cdot)}=0$;

\item
for every continuous linear functional $G$ on $L^{p(\cdot)}(\R)$ there exists a
unique function $g\in L^{p'(\cdot)}(\R)$ such that
\[
G(f)=\int_\R f(x)\overline{g(x)}dx\quad
\ \mbox{for}\
f\in L^{p(\cdot)}(\R)
\]
and the norms $\|G\|$ and $\|g\|_{p'(\cdot)}$ are equivalent.
\end{enumerate}
\end{theorem}
%%%%%%%%%%%%%%%%%%%%%%%%%%%%%%%%%%%%%%%%%%%%%%%%%%%%%%%%%%%%%%%%%%%%%%%%%%%
\begin{corollary}\label{co:duality-N}
Let $p:\R\to[1,\infty]$ be a measurable function satisfying \eqref{eq:exponents}.
For every continuous linear functional $G$ on $L_N^{p(\cdot)}(\R)$ there
exists a unique function $g=(g_1,\dots,g_N)\in L_N^{p'(\cdot)}(\R)$ such
that
%%%
\begin{equation}\label{eq:duality-N}
G(f)=\sum_{\alpha=1}^N\int_\R f_\alpha(x)\overline{g_\alpha(x)}dx=:\langle f,g\rangle
\end{equation}
%%%
for all $f=(f_1,\dots,f_N)\in L_N^{p(\cdot)}(\R)$. The norms of $\|G\|$
and $\|g\|_{L_N^{p'(\cdot)}(\R)}$ are equivalent.
\end{corollary}
%%%%%%%%%%%%%%%%%%%%%%%%%%%%%%%%%%%%%%%%%%%%%%%%%%%%%%%%%%%%%%%%%%%%%%%%%%%
\subsection{Singular integral operators and their adjoints}
For $a\in L_{N\times N}^\infty(\R)$, let $a^*$ denote the complex
conjugate of the transpose matrix function $a^{\mathrm{T}}$.
%%%%%%%%%%%%%%%%%%%%%%%%%%%%%%%%%%%%%%%%%%%%%%%%%%%%%%%%%%%%%%%%%%%%%%%%%%
\begin{lemma}\label{le:adjoint-multiplication}
Let $p:\R\to[1,\infty]$ be a measurable function satisfying \eqref{eq:exponents}.
If $a\in L_{N\times N}^\infty(\R)$, then
\[
(aI)^*=a^*I \in\cB(L_N^{p'(\cdot)}(\R)).
\]
\end{lemma}
%%%%%%%%%%%%%%%%%%%%%%%%%%%%%%%%%%%%%%%%%%%%%%%%%%%%%%%%%%%%%%%%%%%%%%%%%%%
\begin{proof}
Let $\langle\cdot,\cdot\rangle$ be the pairing defined by \eqref{eq:duality-N}
and $f\in L_N^{p(\cdot)}(\R)$, $g\in L_N^{p'(\cdot)}(\R)$. Then
%%%
\begin{align*}
\langle af,g\rangle
&=
\sum_{\alpha=1}^N\int_\R
\left(
\sum_{\beta=1}^N a_{\alpha\beta}(x) f_\beta(x)
\right)
\overline{g_\alpha(x)}\,dx
=
\sum_{\beta=1}^N\int_\R
\left(
\sum_{\alpha=1}^N a_{\alpha\beta}(x) \overline{g_\alpha(x)}
\right)
f_\beta(x)\,dx
\\
&=
\sum_{\alpha=1}^N\int_\R
\left(
\sum_{\beta=1}^N a_{\beta\alpha}(x) \overline{g_\beta(x)}
\right)
f_\alpha(x)\,dx
=
\sum_{\alpha=1}^N\int_\R f_\alpha(x)\overline{
\left(\sum_{\beta=1}^N\overline{a_{\beta\alpha}(x)}g_\beta(x)\right)}\,dx
=
\langle f,a^*g\rangle,
\end{align*}
%%%
which completes the proof in view of Corollary~\ref{co:duality-N}.
\qed
\end{proof}
%%%%%%%%%%%%%%%%%%%%%%%%%%%%%%%%%%%%%%%%%%%%%%%%%%%%%%%%%%%%%%%%%%%%%%%%%%%
\begin{lemma}\label{le:projections}
If $p\in\cE$, then $P,Q\in\cB(L_N^{p(\cdot)}(\R))$ and $P^2=P$, $Q^2=Q$.
\end{lemma}
%%%%%%%%%%%%%%%%%%%%%%%%%%%%%%%%%%%%%%%%%%%%%%%%%%%%%%%%%%%%%%%%%%%%%%%%%%%
\begin{proof}
Since the operators $S$, $P$, and $Q$ are defined elementwise, it is
sufficient to prove the statement for $N=1$. It is well known (see, e.g.,
\cite[formula (3.5)]{D01}) that
\[
S^2\varphi=\varphi\quad\mbox{for}\quad\varphi\in L^2(\R).
\]
In particular, the above formula holds for all $\varphi\in C_c^\infty(\R)$.
Let $f\in L^{p(\cdot)}(\R)$. By Theorem~\ref{th:KR}(a), there exists a sequence
$\{\varphi_n\}_{n=1}^\infty\subset C_c^\infty(\R)$ such that
%%%
\begin{equation}\label{eq:projections-1}
\lim_{n\to\infty}\|f-\varphi_n\|_{p(\cdot)}=0.
\end{equation}
%%%
Since $p\in\cE$, we conclude that $S^2\in\cB(L^{p(\cdot)}(\R))$. Hence
%%%
\begin{equation}\label{eq:projections-2}
\lim_{n\to\infty}\|S^2f-S^2\varphi_n\|_{p(\cdot)}\le\|S^2\|_{\cB(L^{p(\cdot)}(\R))}
\lim_{n\to\infty}\|f-\varphi_n\|_{p(\cdot)}=0.
\end{equation}
%%%
Passing to the limit in the equality $S^2\varphi_n=\varphi_n$ as $n\to\infty$
and taking into account \eqref{eq:projections-1}--\eqref{eq:projections-2},
we arrive at $S^2f=f$ for $f\in L^{p(\cdot)}(\R)$, that is, $S^2=I$ on
$L^{p(\cdot)}(\R)$. This immediately implies that $P^2=P$ and $Q^2=Q$.
\qed
\end{proof}
%%%%%%%%%%%%%%%%%%%%%%%%%%%%%%%%%%%%%%%%%%%%%%%%%%%%%%%%%%%%%%%%%%%%%%%%%%%
\begin{lemma}\label{le:adjoints-SPQ}
If $p\in\cE$, then $p'\in\cE$ and
\[
S^*=S,\quad
P^*=P,\quad
Q^*=Q
\]
belong to $\cB(L_N^{p'(\cdot)}(\R))$.
\end{lemma}
%%%%%%%%%%%%%%%%%%%%%%%%%%%%%%%%%%%%%%%%%%%%%%%%%%%%%%%%%%%%%%%%%%%%%%%%%%%
\begin{proof}
Since the operators $S$, $P$, and $Q$ are defined elementwise on $L^{p(\cdot)}(\R)$,
it is sufficient to prove the statement for $N=1$. It is well known
that for $\varphi,\psi\in L^2(\R)$,
\[
\int_\R(S\varphi)(x)\overline{\psi(x)}\,dx=\int_\R\varphi(x)\overline{(S\psi)(x)}\,dx
\]
(see, e.g., \cite[formula (3.6)]{D01}). In particular, this equality holds for
all $\varphi,\psi\in C_c^\infty(\R)$. This means that $S$ is a self-adjoint
and densely defined operator on $L^{p(\cdot)}(\R)$ and $L^{p'(\cdot)}(\R)$ (see
Theorem~\ref{th:KR}(a)). By the standard argument (see
\cite[Chap.~III, Section~5.5]{K95}), one can show that $S=S^*
\in\cB(L^{p'(\cdot)}(\R))$ because
$S\in\cB(L^{p(\cdot)}(\R))$. This yields $p'\in\cE$ and also
the equalities
\[
P^*=(I+S)^*/2=(I+S)/2=P,
\quad
Q^*=(I-S)^*/2=(I-S)/2=Q,
\]
which finishes the proof.
\qed
\end{proof}
%%%%%%%%%%%%%%%%%%%%%%%%%%%%%%%%%%%%%%%%%%%%%%%%%%%%%%%%%%%%%%%%%%%%%%%%%%%
From Lemmas~\ref{le:adjoint-multiplication} and \ref{le:adjoints-SPQ}
we immediately get the following.
%%%%%%%%%%%%%%%%%%%%%%%%%%%%%%%%%%%%%%%%%%%%%%%%%%%%%%%%%%%%%%%%%%%%%%%%%%%
\begin{corollary}\label{co:adjoints-sio}
If $p\in\cE$  and $a,b\in L_{N\times N}^\infty(\R)$, then
\[
(aP+bQ)^*=Pa^*I+Qb^*I\in\cB(L_N^{p'(\cdot)}(\R)).
\]
\end{corollary}
%%%%%%%%%%%%%%%%%%%%%%%%%%%%%%%%%%%%%%%%%%%%%%%%%%%%%%%%%%%%%%%%%%%%%%%%%%%
The proof of the next statement is a matter of a
straightforward calculation and application of
Lemma~\ref{le:projections} when necessary.
%%%%%%%%%%%%%%%%%%%%%%%%%%%%%%%%%%%%%%%%%%%%%%%%%%%%%%%%%%%%%%%%%%%%%%%%%%%
\begin{lemma}\label{le:relations-sio}
If $p\in\cE$ and $a\in L_{N\times N}^\infty(\R)$, then
%%%
\begin{equation}\label{eq:relations-sio-1}
(I\pm PaQ)^{-1}=I\mp PaQ,
\quad
(I\pm QaP)^{-1}=I\mp QaP,
\end{equation}
%%%
and
%%%
\[
PaI+Q =(I+PaQ)(aP+Q)(I-QaP),
\quad
P+QaI =(I+QaP)(P+aQ)(I-PaQ).
\]
\end{lemma}
%%%%%%%%%%%%%%%%%%%%%%%%%%%%%%%%%%%%%%%%%%%%%%%%%%%%%%%%%%%%%%%%%%%%%%%%%%%
\subsection{Compact operators and convergence of sequences of operators}
\begin{lemma}[{see, e.g., \cite[Lemma~1.4.7]{RSS11}}]
\label{le:compact-convergence}
Let $X$ be a Banach space. Suppose $A,B\in\cB(X)$, and $A_n,B_n\in\cB(X)$
for all $n\in\N$. If $K\in\cK(X)$ and if $A_n\to A$ and
$B_n^*\to B^*$ strongly as $n\to\infty$, then $\|A_nKB_n-AKB\|_{\cB(X)}\to 0$ as
$n\to\infty$.
\end{lemma}
%%%%%%%%%%%%%%%%%%%%%%%%%%%%%%%%%%%%%%%%%%%%%%%%%%%%%%%%%%%%%%%%%%%%%%%%%%%
Let $\chi_E$ be the characteristic function of a set $E\subset\R$.
%%%%%%%%%%%%%%%%%%%%%%%%%%%%%%%%%%%%%%%%%%%%%%%%%%%%%%%%%%%%%%%%%%%%%%%%%%%
\begin{lemma}\label{le:convergence-compact-multiplication}
Let $p:\R\to[1,\infty]$ be a measurable function satisfying
\eqref{eq:exponents}. For $n\in\N$ and $x\in\R$, put
\[
\psi_n(x):=1-\chi_{[-1,1]}(x/n).
\]
\begin{enumerate}[(a)]
\item
The sequence $\{\psi_n I\}_{n=1}^\infty$ converges strongly to the zero operator
on $L^{p(\cdot)}(\R)$ and on $L^{p'(\cdot)}(\R)$ as $n\to\infty$.

\item
If $K\in\cK(L^{p(\cdot)}(\R))$, then
\[
\lim_{n\to\infty}\|K\psi_nI\|_{\cB(L^{p(\cdot)}(\R))}=0.
\]
\end{enumerate}
\end{lemma}
%%%%%%%%%%%%%%%%%%%%%%%%%%%%%%%%%%%%%%%%%%%%%%%%%%%%%%%%%%%%%%%%%%%%%%%%%%%
\begin{proof}
(a) If $1<\operatornamewithlimits{ess\,inf}\limits_{x\in\R}p(x)$, then
$\operatornamewithlimits{ess\,sup}\limits_{x\in\R}p'(x)<\infty$. Therefore, by
Theorem~\ref{th:KR}(a)--(b), it is sufficient to prove that
%%%
\begin{equation}\label{eq:ccm-1}
\lim_{n\to\infty}I_{p(\cdot)}(\psi_n f)=0
\quad\text{for all}\quad
f\in C_c^\infty(\R).
\end{equation}
%%%

Suppose $f\in C_c^\infty(\R)$. Then there exists $n_0\in\N$ such that
$\operatorname{supp} f\subset[-n_0,n_0]$. Then for all $n\ge n_0$,
%%%
\[
I_{p(\cdot)}(\psi_n f)
=
\int_\R|\big(1-\chi_{[-1,1]}(x/n)\big)f(x)|^{p(x)}dx
=
\int_\R|\chi_{\R\setminus[-n,n]}(x)f(x)|^{p(x)}dx
=
\int_{\R\setminus[-n,n]}|f(x)|^{p(x)}dx=0.
\]
%%%
Thus $I_{p(\cdot)}(\psi_n f)=0$ for all $n\ge n_0$, which finishes the proof of
\eqref{eq:ccm-1}. Part (a) is proved.

\medskip
(b) From Theorem~\ref{th:KR}(c) it follows that $(\psi_n I)^*=\psi_n I\in\cB(L^{p'(\cdot)}(\R))$.
By part~(a), the sequence $\{(\psi_n I)^*\}_{n=1}^\infty$ converges strongly to
the zero operator. It remains to apply Lemma~\ref{le:compact-convergence}.
\qed
\end{proof}
%%%%%%%%%%%%%%%%%%%%%%%%%%%%%%%%%%%%%%%%%%%%%%%%%%%%%%%%%%%%%%%%%%%%%%%%%%%
\subsection{Important property of slowly oscillating functions}
The following statement is proved by analogy with \cite[Proposition 4(ii)]{BKS04}.
%%%%%%%%%%%%%%%%%%%%%%%%%%%%%%%%%%%%%%%%%%%%%%%%%%%%%%%%%%%%%%%%%%%%%%%%%%%
\begin{lemma}\label{le:SO-uniform}
Let $f\in SO$. Suppose $\{h_k\}_{k=1}^\infty\subset\R$ is a sequence tending
to $+\infty$ {\rm(}resp. to $-\infty${\rm)} and such that the limit
%%%
\begin{equation}\label{eq:SO-uniform-1}
\lim_{k\to\infty} f(h_k)=:g
\end{equation}
%%%
exist. Then for every $R>0$,
%%%
\begin{equation}\label{eq:SO-uniform-2}
\lim_{k\to\infty}\sup_{x\in[-R,R]}|f(x+h_k)-g|=0.
\end{equation}
\end{lemma}
%%%%%%%%%%%%%%%%%%%%%%%%%%%%%%%%%%%%%%%%%%%%%%%%%%%%%%%%%%%%%%%%%%%%%%%%%%%
\begin{proof}
For every $k\in\N$,
%%%
\begin{align}
\sup_{x\in[-R,R]}|f(x+h_k)-g|
&\le
\sup_{x\in[-R,R]}|f(x+h_k)-f(h_k)|+|f(h_k)-g|
\nonumber
\\
&\le
\sup_{x,y\in[h_k-R,h_k+R]}|f(x)-f(y)|+|f(h_k)-g|.
\label{eq:SO-uniform-3}
\end{align}
%%%
Let for definiteness $\lim_{k\to\infty}h_k=-\infty$. Then there exists a $k_0\in\N$
such that $h_k\le -3R$ for all $k\ge k_0$. Therefore $2(h_k+R)\le h_k-R$ and
\[
[2(h_k+R),h_k+R]\supset [h_k-R,h_k+R].
\]
Thus for $k\ge k_0$,
%%%
\begin{align}
\sup_{x,y\in[h_k-R,h_k+R]} &|f(x)-f(y)|
\le
\sup_{x,y\in[2(h_k+R),h_k+R]}|f(x)-f(y)|
\nonumber
\\
&\le
\operatorname{osc}\big(f,[2(h_k+R),h_k+R]\cup[-(h_k+R),-2(h_k+R)]\big).
\label{eq:SO-uniform-4}
\end{align}
%%%
Since $f\in SO$, the latter oscillation tends to zero as $k\to\infty$.
Combining this observation with \eqref{eq:SO-uniform-1} and
\eqref{eq:SO-uniform-3}--\eqref{eq:SO-uniform-4}, we arrive at \eqref{eq:SO-uniform-2}.
\qed
\end{proof}
%%%%%%%%%%%%%%%%%%%%%%%%%%%%%%%%%%%%%%%%%%%%%%%%%%%%%%%%%%%%%%%%%%%%%%%%%%%
\subsection{Lemma on an implicit sequence}
We will need the following result from Elementary Calculus.
Put $\R_+:=(0,+\infty)$ and $\R_-:=(-\infty,0)$.
%%%%%%%%%%%%%%%%%%%%%%%%%%%%%%%%%%%%%%%%%%%%%%%%%%%%%%%%%%%%%%%%%%%%%%%%%%%
\begin{lemma}\label{le:implicit}
Let $F:\R_+\times (\N\cup\{\infty\})\to\R_+$ be a function such that
%%%
\begin{enumerate}[(i)]
\item
for every $k\in\N\cup\{\infty\}$, the function $F(\cdot,k)$ is continuous
and strictly decreasing;

\item
for every $\lambda\in\R_+$,
%%%
\begin{equation}\label{eq:implicit-1}
\lim_{k\to\infty}F(\lambda,k)=F(\lambda,\infty).
\end{equation}
\end{enumerate}
%%%
If $F(\lambda_\infty,\infty)=1$ for some $\lambda_\infty\in\R_+$, then there
exists a number $k_0\in\N$ and a unique sequence $\{\lambda(k)\}_{k=k_0}^\infty$
such that $F(\lambda(k),k)=1$ for all $k\ge k_0$ and
%%%
\begin{equation}\label{eq:implicit-2}
\lim_{k\to\infty}\lambda(k)=\lambda_\infty.
\end{equation}
\end{lemma}
%%%%%%%%%%%%%%%%%%%%%%%%%%%%%%%%%%%%%%%%%%%%%%%%%%%%%%%%%%%%%%%%%%%%%%%%%%%%
\begin{proof}
The proof is developed by analogy with the proof of the lemma from
\cite[Section~41.1]{Ku03}.

Let $\eps\in(0,\lambda_\infty/2]$. Since $F(\cdot,\infty)$ is strictly decreasing,
%%%
\begin{equation}\label{eq:implicit-3}
F(\lambda_\infty+\eps,\infty)<F(\lambda_\infty,\infty)=1<F(\lambda_\infty-\eps,\infty).
\end{equation}
%%%
From \eqref{eq:implicit-1} it follows that there exist $k_+(\eps),k_-(\eps)\in\N$
such that
%%%
\begin{equation}\label{eq:implicit-4}
|F(\lambda_\infty+\eps,\infty)-F(\lambda_\infty+\eps,k)|<\frac{1-F(\lambda_\infty+\eps,\infty)}{2}
\end{equation}
%%%
for $k\ge k_+(\eps)$ and
%%%
\begin{equation}\label{eq:implicit-5}
|F(\lambda_\infty-\eps,\infty)-F(\lambda_\infty-\eps,k)|<\frac{F(\lambda_\infty-\eps,\infty)-1}{2}
\end{equation}
%%%
for $k\ge k_-(\eps)$. Let
\[
k_0(\eps):=\max\{k_-(\eps),k_+(\eps)\},\quad
k_0:=k_0(\lambda_\infty/2).
\]
Taking into account \eqref{eq:implicit-3}, we obtain from
\eqref{eq:implicit-4}--\eqref{eq:implicit-5} that
%%%
\begin{align*}
F(\lambda_\infty+\eps,k)
&<
\frac{1-F(\lambda_\infty+\eps,\infty)}{2}+F(\lambda_\infty+\eps,\infty)
=\frac{1+F(\lambda_\infty+\eps,\infty)}{2}
<1,
\\
F(\lambda_\infty-\eps,k)
&>
F(\lambda_\infty-\eps,\infty)-\frac{F(\lambda_\infty-\eps,\infty)-1}{2}
=
\frac{1+F(\lambda_\infty-\eps,\infty)}{2}
>1.
\end{align*}
%%%
Thus, for all $k\ge k_0(\eps)$,
%%%
\begin{equation}\label{eq:implicit-6}
F(\lambda_\infty+\eps,k)<1<F(\lambda_\infty-\eps,k).
\end{equation}
%%%
Since $F(\cdot,k)$ is continuous in the first variable for every fixed $k$,
from \eqref{eq:implicit-6} we see, by the Bolzano-Cauchy intermediate value
theorem, that there exists a $\lambda(k)$ such that $F(\lambda(k),k)=1$
and
%%%
\begin{equation}\label{eq:implicit-7}
\lambda_\infty-\eps<\lambda(k)<\lambda_\infty+\eps.
\end{equation}
%%%
The value $\lambda(k)$ is unique for every $k$ because $F(\cdot,k)$
is strictly decreasing. Thus, for every $\eps\in(0,\infty/2]$, there
exists a number $k_0(\eps)\in\N$ such that for all $k\ge k_0(\eps)$,
inequality \eqref{eq:implicit-7} holds, which implies \eqref{eq:implicit-2}.
\qed
\end{proof}
%%%%%%%%%%%%%%%%%%%%%%%%%%%%%%%%%%%%%%%%%%%%%%%%%%%%%%%%%%%%%%%%%%%%%%%%%%%

%%%%%%%%%%%%%%%%%%%%%%%%%%%%%%%%%%%%%%%%%%%%%%%%%%%%%%%%%%%%%%%%%%%%%%%%%%%
\section{Norms of translations of decaying continuous functions}
\label{sec:key}
%%%%%%%%%%%%%%%%%%%%%%%%%%%%%%%%%%%%%%%%%%%%%%%%%%%%%%%%%%%%%%%%%%%%%%%%%%%
\subsection{Technical lemma}
We start with the following technical statement.
%%%%%%%%%%%%%%%%%%%%%%%%%%%%%%%%%%%%%%%%%%%%%%%%%%%%%%%%%%%%%%%%%%%%%%%%%%%
\begin{lemma}\label{le:technical}
Suppose $p:\R\to(1,\infty)$ belongs to $SO$ and satisfies
\eqref{eq:exponents}. Let $\{h_k\}_{k=1}^\infty\subset\R$ be a
sequence tending to $+\infty$ {\rm(}resp. to
$-\infty${\rm)} and such that the limit
\[
\lim\limits_{k\to\infty}p(h_k)=:q
\]
exists.  Suppose $R>0$ and $\{w_k\}_{k=1}^\infty\subset C(\R)$ is
a sequence which converges pointwise to a function $w\in C(\R)$ on the
segment $[-R,R]$. If there are positive constants $C_1<C_2$ and a measurable
set $\Delta\subset[-R,R]$ such that for all sufficiently large $k$ and all
$x\in[-R,R]\setminus\Delta$,
%%%
\begin{equation}\label{eq:technical-1}
C_1\le w_k(x)\le C_2,\quad C_1\le w(x)\le C_2,
\end{equation}
%%%
then for every $\lambda\in\R_+$,
%%%
\begin{equation}\label{eq:technical-2}
\lim_{k\to\infty}\int_{[-R,R]\setminus\Delta}\left|\frac{w_k(x)}{\lambda}\right|^{p(x+h_k)}dx
=
\int_{[-R,R]\setminus\Delta}\left|\frac{w(x)}{\lambda}\right|^{q}dx.
\end{equation}
\end{lemma}
%%%%%%%%%%%%%%%%%%%%%%%%%%%%%%%%%%%%%%%%%%%%%%%%%%%%%%%%%%%%%%%%%%%%%%%%%%%
\begin{proof}
The proof is based on the Lebesgue bounded convergence theorem
(see, e.g., \cite[Theorem~10.29]{A74}). Let us show that for all
$\lambda\in\R_+$ and all $x\in[-R,R]\setminus\Delta$,
%%%
\begin{equation}\label{eq:technical-3}
\lim_{k\to\infty}\left|\frac{w_k(x)}{\lambda}\right|^{p(x+h_k)}=
\left|\frac{w(x)}{\lambda}\right|^q.
\end{equation}
%%%
By the mean value theorem,
\[
\begin{split}
\left|\frac{w_k(x)}{\lambda}\right|^{p(x+h_k)}-\left|\frac{w(x)}{\lambda}\right|^q
&=
\exp\left(p(x+h_k)\log\left|\frac{w_k(x)}{\lambda}\right|\right)
-
\exp\left(q\log\left|\frac{w(x)}{\lambda}\right|\right)
\\
&=e^\xi
\left(
p(x+h_k)\log\left|\frac{w_k(x)}{\lambda}\right|
-
q\log\left|\frac{w(x)}{\lambda}\right|
\right),
\end{split}
\]
where $\xi$ is some real number between
\[
p(x+h_k)\log\left|\frac{w_k(x)}{\lambda}\right|
\quad\mbox{and}\quad
q\log\left|\frac{w(x)}{\lambda}\right|.
\]
Taking into account that there exists a $k_0\in\N$ such that for all $k\ge k_0$
inequalities \eqref{eq:technical-1} are fulfilled, we have
\[
p(x+h_k)\log\left|\frac{w_k(x)}{\lambda}\right|
\le
p(x+h_k)\log\frac{C_2}{\lambda}
\le
p(x+h_k)\left|\log\frac{C_2}{\lambda}\right|
\le
p_+\left|\log\frac{C_2}{\lambda}\right|
\]
and
\[
q\log\left|\frac{w(x)}{\lambda}\right|
\le
q\log\frac{C_2}{\lambda}
\le
q\left|\log\frac{C_2}{\lambda}\right|
\le
p_+\left|\log\frac{C_2}{\lambda}\right|.
\]
Hence
\[
\xi\le p_+\left|\log\frac{C_2}{\lambda}\right|
\quad\mbox{and}\quad
e^\xi\le\exp\left(p_+\left|\log\frac{C_2}{\lambda}\right|\right)=:C_3.
\]
Then for all $k\ge k_0$,
%%%
\begin{align}
\left|
\left|\frac{w_k(x)}{\lambda}\right|^{p(x+h_k)}
-
\left|\frac{w(x)}{\lambda}\right|^q
\right|
\le&
C_3\left|
p(x+h_k)\log\left|\frac{w_k(x)}{\lambda}\right|
-
q\log\left|\frac{w(x)}{\lambda}\right|
\right|
\nonumber
\\
\le&
C_3|p(x+h_k)-q|\,\left|\log\left|\frac{w_k(x)}{\lambda}\right|\right|
+
C_3q\left|\log\left|\frac{w_k(x)}{\lambda}\right|-\log\left|\frac{w(x)}{\lambda}\right|\right|.
\label{eq:technical-4}
\end{align}
%%%
Further, we have
\[
\begin{split}
\log\left|\frac{w_k(x)}{\lambda}\right|
&\le
\log\frac{C_2}{\lambda}
\le
\left|\log\frac{C_2}{\lambda}\right|
\le
\max\left\{\left|\log\frac{C_1}{\lambda}\right|,\left|\log\frac{C_2}{\lambda}\right|\right\},
\\
\log\left|\frac{w_k(x)}{\lambda}\right|
&\ge
\log\frac{C_1}{\lambda}
\ge
-\left|\log\frac{C_1}{\lambda}\right|
\ge
-\max\left\{\left|\log\frac{C_1}{\lambda}\right|,\left|\log\frac{C_2}{\lambda}\right|\right\}.
\end{split}
\]
Therefore, for all $k\ge k_0$,
%%%
\begin{equation}\label{eq:technical-5}
\left|
\log\left|\frac{w_k(x)}{\lambda}\right|
\right|
\le
\max\left\{\left|\log\frac{C_1}{\lambda}\right|,\left|\log\frac{C_2}{\lambda}\right|\right\}
=:C_4<\infty.
\end{equation}
%%%
Applying the main value theorem once again, we see that
\[
\log\left|\frac{w_k(x)}{\lambda}\right|-\log\left|\frac{w(x)}{\lambda}\right|
=\frac{1}{\zeta}\big(|w_k(x)|-|w(x)|\big),
\]
where $\zeta$ is some number between $|w_k(x)|$ and $|w(x)|$. Hence $\zeta\in[C_1,C_2]$.
Then for all $k\ge k_0$,
%%%
\begin{equation}\label{eq:technical-6}
\left|\log\left|\frac{w_k(x)}{\lambda}\right|-\log\left|\frac{w(x)}{\lambda}\right|\right|
\le\frac{1}{C_1}\big||w_k(x)|-|w(x)|\big|
\le\frac{1}{C_1}|w_k(x)-w(x)|.
\end{equation}
%%%
Combining \eqref{eq:technical-4}--\eqref{eq:technical-6}, we arrive at
%%%
\begin{equation}\label{eq:technical-7}
\left|
\left|\frac{w_k(x)}{\lambda}\right|^{p(x+h_k)}
-
\left|\frac{w(x)}{\lambda}\right|^q
\right|
\le
C_3C_4|p(x+h_k)-q|+\frac{C_3q}{C_1}|w_k(x)-w(x)|
\end{equation}
%%%
for all $k\ge k_0$. From Lemma~\ref{le:SO-uniform} it follows that
%%%
\begin{equation}\label{eq:technical-8}
\lim_{k\to\infty}|p(x+h_k)-q|=0.
\end{equation}
%%%
But it is given that
%%%
\begin{equation}\label{eq:technical-9}
\lim_{k\to\infty}|w_k(x)-w(x)|=0.
\end{equation}
%%%
Thus, from inequality \eqref{eq:technical-7} and equalities
\eqref{eq:technical-8}--\eqref{eq:technical-9} we immediately get
\eqref{eq:technical-3}.

Further, for every $x\in [-R,R]\setminus\Delta$ and $k\ge k_0$,
\[
\left|\frac{w_k(x)}{\lambda}\right|^{p(x+h_k)}
\le
\left(\frac{C_2}{\lambda}\right)^{p(x+h_k)}
\le
\left(\max\left\{1,\frac{C_2}{\lambda}\right\}\right)^{p(x+h_k)}
\le
\left(\max\left\{1,\frac{C_2}{\lambda}\right\}\right)^{p_+}
\]
because $p(x+h_k)\le p_+$. Thus, the sequence $|w_k(x)/\lambda|^{p(x+h_k)}$
is uniformly bounded and converges pointwise to $|w(x)/\lambda|^q$. By the
Lebesgue bounded convergence theorem, this yields \eqref{eq:technical-2}.
\qed
\end{proof}
%%%%%%%%%%%%%%%%%%%%%%%%%%%%%%%%%%%%%%%%%%%%%%%%%%%%%%%%%%%%%%%%%%%%%%%%%%%
\subsection{Key lemma}
The key to the proof of Theorem~\ref{th:main} is the following
generalization of the one-dimensional version of
\cite[Proposition~6.3]{RS08}. Note that conditions on $p$ imposed in
\cite{RS08} imply that $p\in C(\dot{\R})$. For the
readers' convenience, we provide here a detailed proof in our more
general situation, though the outline remains more or less the same
as in \cite{RS08}.
%%%%%%%%%%%%%%%%%%%%%%%%%%%%%%%%%%%%%%%%%%%%%%%%%%%%%%%%%%%%%%%%%%%%%%%%%%%
\begin{lemma}\label{le:key}
Suppose $p:\R\to(1,\infty)$ belongs to $SO$ and satisfies
\eqref{eq:exponents}. Let $\{h_k\}_{k=1}^\infty\subset\R$ be a
sequence tending to  $+\infty$ {\rm(}resp. to
$-\infty${\rm)} and such that the limit
\[
\lim\limits_{k\to\infty}p(h_k)=:q
\]
exists. Suppose $w\in C(\R)$ and
$\{w_k\}_{k=1}^\infty\subset C(\R)$ are such that
%%%
\begin{enumerate}[(i)]
\item
for all $x\in\R$,
\[
\lim_{k\to\infty}w_k(x)=w(x),
\]
and this convergence is uniform on each closed segment $J\subset\R_+$;

\item
there exists a constant $C>0$ such that for all $k\in\N$
and $x\in\R$,
\[
|w(x)|\le \frac{C}{1+|x|},\quad |w_k(x)|\le \frac{C}{1+|x|}.
\]
\end{enumerate}
%%%
Then
\begin{equation}\label{eq:key-1}
\lim_{k\to\infty}\|V_{h_k}w_k\|_{p(\cdot)}=\|w\|_q.
\end{equation}
\end{lemma}
%%%%%%%%%%%%%%%%%%%%%%%%%%%%%%%%%%%%%%%%%%%%%%%%%%%%%%%%%%%%%%%%%%%%%%%%%%%
\begin{proof}
For $\lambda>0$ and $k\in\N$, put
\[
F(\lambda,k)
:=
\int_\R\left|\frac{(V_{h_k}w_k)(x)}{\lambda}\right|^{p(x)}dx
=
\int_\R\left|\frac{w_k(x)}{\lambda}\right|^{p(x+h_k)}dx,
\quad
F(\lambda,\infty):=\int_\R\left|\frac{w(x)}{\lambda}\right|^qdx=\lambda^{-q}\|w\|_q^q.
\]
First, let us show that for every $\lambda>0$,
%%%
\begin{equation}\label{eq:key-2}
\lim_{k\to\infty}F(\lambda,k)=F(\lambda,\infty).
\end{equation}
%%%
Fix some numbers $R>0$ and $\delta>0$. We will specify the choice of $R$ and
$\delta$ later. Consider the (possibly empty) set
%%%
\begin{equation}\label{eq:key-3}
\Delta_\delta:=\big\{x\in[-R,R]:\ |w(x)|\le 2\delta\big\}
\end{equation}
%%%
and put
\[
\begin{array}{lll}
\displaystyle
T_R(\lambda,k)
:=
\int_{|x|>R}\left|\frac{w_k(x)}{\lambda}\right|^{p(x+h_k)}dx,
&&
\displaystyle
T_R(\lambda,\infty)
:=
\int_{|x|>R}\left|\frac{w(x)}{\lambda}\right|^q dx,
\\[3mm]
\displaystyle
L_{\delta,R}(\lambda,k)
:=
\int_{\Delta_\delta}\left|\frac{w_k(x)}{\lambda}\right|^{p(x+h_k)}dx,
&&
\displaystyle
L_{\delta,R}(\lambda,\infty)
:=
\int_{\Delta_\delta}\left|\frac{w(x)}{\lambda}\right|^q dx,
\end{array}
\]
%%%
and
\[
D_{\delta,R}(\lambda,k):=
\left|\int_{[-R,R]\setminus\Delta_\delta}
\left|\frac{w_k(x)}{\lambda}\right|^{p(x+h_k)}dx
-
\int_{[-R,R]\setminus\Delta_\delta}
\left|\frac{w(x)}{\lambda}\right|^q dx
\right|.
\]
Here ``$T$" is for ``tail", ``$L$" is for ``little", and ``$D$" is for ``difference".
It is clear that
%%%
\begin{equation}\label{eq:key-4}
|F(\lambda,k)-F(\lambda,\infty)|
\le
T_R(\lambda,k)+T_R(\lambda,\infty)
+
L_{\delta,R}(\lambda,k)+L_{\delta,R}(\lambda,\infty)+
D_{\delta,R}(\lambda,k).
\end{equation}
%%%
Fix $\eps>0$. First we will show that it is possible to choose $R$ so large
that for $k\in\N$,
%%%
\begin{equation}\label{eq:key-5}
T_R(\lambda,k)+T_R(\lambda,\infty)<\eps/3.
\end{equation}
%%%
Let for the moment $R\ge C/\lambda$. Then from \eqref{eq:exponents} and
hypothesis (ii) we obtain
\[
\left|\frac{w_k(x)}{\lambda}\right|^{p(x+h_k)}
\le
\left(\frac{C}{\lambda|x|}\right)^{p(x+h_k)}
\le
\left(\frac{C}{\lambda|x|}\right)^{p_-}
\quad\mbox{for}\quad
|x|\ge R.
\]
Then for $\lambda>0$, $k\in\N$, and $R\ge C/\lambda$,
%%%
\begin{equation}\label{eq:key-6}
T_R(\lambda,k)
\le
\int_{|x|>R}\left(\frac{C}{\lambda|x|}\right)^{p_-}dx
=
2\left(\frac{C}{\lambda}\right)^{p_-}\int_R^{+\infty}\frac{dx}{x^{p_-}}
=
\frac{2}{p_--1}\left(\frac{C}{\lambda}\right)^{p_-}R^{1-p_-}
\end{equation}
%%%
and analogously
%%%
\begin{equation}\label{eq:key-7}
T_R(\lambda,\infty)
\le
\frac{2}{p_--1}\left(\frac{C}{\lambda}\right)^{p_-}R^{1-p_-}
\end{equation}
%%%
(recall that $q\ge p_->1$). We choose $R$ as the solution of the equation
%%%
\begin{equation}\label{eq:key-8}
\frac{4}{p_--1}\left(\frac{C}{\lambda}\right)^{p_-}R^{1-p_-}=\frac{\eps}{6}.
\end{equation}
%%%
Then from inequalities \eqref{eq:key-6}--\eqref{eq:key-7} it follows that inequality
\eqref{eq:key-5} holds.

It remains to show that for so chosen $R$ one has $R\ge C/\lambda$ whenever
$\eps$ is sufficiently small. Indeed, from \eqref{eq:key-8} we obtain
%%%
\begin{equation}\label{eq:key-9}
R=\left(\frac{24}{p_--1}\right)^{1/(p_--1)}\left(\frac{C}{\lambda}\right)^{p_-/(p_--1)}
\left(\frac{1}{\eps}\right)^{1/(p_--1)},
\end{equation}
%%%
and $R\ge C/\lambda$ is equivalent to
\[
\left(\frac{24}{p_--1}\right)^{1/(p_--1)}\frac{C}{\lambda}\ge \eps^{1/(p_--1)}.
\]
That is, if
\[
0<\eps\le\frac{24}{p_--1}\left(\frac{C}{\lambda}\right)^{p_--1}=:\eps_1,
\]
then $R$ given by \eqref{eq:key-9} satisfies $R\ge C/\lambda$ and inequality \eqref{eq:key-5}
holds.

Now we will choose $\delta>0$ and $k_0\in\N$ such that for $k\ge k_0$,
%%%
\begin{equation}\label{eq:key-10}
L_{\delta,R}(\lambda,k)+L_{\delta,R}(\lambda,\infty)<\eps/3.
\end{equation}
%%%
Let for the moment $\delta$ is so that $3\delta/\lambda\le 1$. For $R$ and
$\delta$, by hypothesis (i), there exists a $k_0:=k_0(\eps)=k_0(\delta,R)\in\N$
such that for all $x\in[-R,R]$ and all $k\ge k_0$,
\[
|w_k(x)-w(x)|<\delta.
\]
Hence, for all $k\ge k_0$,
%%%
\begin{equation}\label{eq:key-11}
|w(x)|-\delta\le|w_k(x)|\le|w(x)|+\delta.
\end{equation}
%%%
From \eqref{eq:key-3} and \eqref{eq:key-11} we see that for $k\ge k_0$ and $x\in\Delta_\delta$,
\[
\left|\frac{w(x)}{\lambda}\right|\le\frac{2\delta}{\lambda},
\quad
\left|\frac{w_k(x)}{\lambda}\right|\le\frac{3\delta}{\lambda}.
\]
Hence, taking into account that $p(x+h_k)>1$ and $q>1$, we have
for $k\ge k_0$,
%%%
\begin{align}
\label{eq:key-12}
&
L_{\delta,R}(\lambda,k)
\le
\int_{\Delta_\delta}\left(\frac{3\delta}{\lambda}\right)^{p(x+h_k)}dx
\le
\frac{3\delta}{\lambda}\int_{\Delta_\delta}dx
\le
\frac{6\delta R}{\lambda},
\\
\label{eq:key-13}
&
L_{\delta,R}(\lambda,\infty)
\le
\int_{\Delta_\delta}\left(\frac{2\delta}{\lambda}\right)^q dx
\le
\frac{2\delta}{\lambda}\int_{\Delta_\delta}dx
\le
\frac{4\delta R}{\lambda}.
\end{align}
%%%
Let us choose $\delta$ as the solution of the equation
%%%
\begin{equation}\label{eq:key-14}
\frac{10\delta R}{\lambda}=\frac{\eps}{6}.
\end{equation}
%%%
Then from inequalities \eqref{eq:key-12}--\eqref{eq:key-13} it follows that inequality
\eqref{eq:key-10} is fulfilled for all $k\ge k_0$.

It remains to show that we can guarantee that $3\delta/\lambda\le 1$ whenever
$\eps$ is sufficiently small. Indeed, from \eqref{eq:key-9} and \eqref{eq:key-14}
we see that
\[
\frac{3\delta}{\lambda}=\frac{\eps}{20R}=
\frac{\eps}{20}\left(\frac{p_--1}{24}\right)^{1/(p_--1)}\left(\frac{\lambda}{C}\right)^{p_-/(p_--1)}\eps^{1/(p_--1)}\le 1
\]
is equivalent to
\[
\eps^{2/(p_--1)}\le 20\left(\frac{24}{p_--1}\right)^{1/(p_--1)}\left(\frac{C}{\lambda}\right)^{p_-/(p_--1)}.
\]
That is, if
\[
0<\eps\le 20^{(p_--1)/2}\left(\frac{24}{p_--1}\right)^{1/2}\left(\frac{C}{\lambda}\right)^{p_-/2}=:\eps_2,
\]
then $3\delta/\lambda\le 1$. Thus, if $\eps\in(0,\min\{\eps_1,\eps_2\})$,
then we can choose $R>0$ by \eqref{eq:key-9}, $\delta>0$ as the solution of \eqref{eq:key-14},
and then choose a $k_0=k_0(\delta,R)$ such that for all $k\ge k_0$, inequalities
\eqref{eq:key-5} and \eqref{eq:key-10} are fulfilled. From \eqref{eq:key-4}, \eqref{eq:key-5},
and \eqref{eq:key-10} we get
%%%
\begin{equation}\label{eq:key-15}
|F(\lambda,k)-F(\lambda,\infty)|\le 2\eps/3+D_{\delta,R}(\lambda,k)
\quad\mbox{for}\quad k\ge k_0.
\end{equation}
%%%
From \eqref{eq:key-3} and \eqref{eq:key-11} it follows that for $x\in[-R,R]\setminus\Delta_\delta$
and $k\ge k_0$,
\[
2\delta<|w(x)|\le C,
\quad
\delta<|w(x)|\le C.
\]
From Lemma~\ref{le:technical} we deduce that there exists $k_1(\eps)\ge k_0$
such that
%%%
\begin{equation}\label{eq:key-16}
D_{\delta,R}(\lambda,k)<\eps/3\quad\mbox{for}\quad k\ge k_1(\eps).
\end{equation}
%%%
Combining \eqref{eq:key-15} and \eqref{eq:key-16}, we see that for $\eps>0$ there exists a
$k_1(\eps)\in\N$ such that for all $k\ge k_1(\eps)$,
\[
|F(\lambda,k)-F(\lambda,\infty)|<\eps,
\]
which finishes the proof of \eqref{eq:key-2}.

If the limit function $w$ is equal to zero identically on $\R$, then from
equality \eqref{eq:key-2} we have
\[
\lim_{k\to\infty}I_{p(\cdot)}(V_{h_k}w_k)=0.
\]
Then from Theorem~\ref{th:KR}(b) we obtain that
\[
\lim_{k\to\infty}\|V_{h_k}w_k\|_{p(\cdot)}=0=\|w\|_q,
\]
which finishes the proof of the lemma in the case $\|w\|_q=0$.

Assume now that $\|w\|_q>0$. Then, obviously, the function
$F(\lambda,\infty)=\lambda^{-q}\|w\|_q$ is strictly decreasing and continuous
in $\lambda\in\R$. Moreover,
%%%
\begin{equation}\label{eq:key-17}
F(\|w\|_q,\infty)=1.
\end{equation}
%%%
Without loss of generality we may assume that all functions $w_k$ are not
identically zero on $\R$. Let us show that for each $k\in\N$, the function
$F(\lambda,k)$ is strictly decreasing and continuous with respect to
$\lambda\in\R_+$. Clearly, for every $k\in\N$, $x\in\R$, and $\lambda\in\R_+$,
\[
\frac{\partial}{\partial\lambda}\left|\frac{w_k(x)}{\lambda}\right|^{p(x+h_k)}
=
-\frac{p(x+h_k)}{\lambda}\left|\frac{w_k(x)}{\lambda}\right|^{p(x+h_k)}.
\]
Let $[\alpha,\beta]\subset\R_+$ be some segment. It is not difficult to see
that for all $\lambda\in[\alpha,\beta]$,
\[
\left|\frac{\partial}{\partial\lambda}\left|\frac{w_k(x)}{\lambda}\right|^{p(x+h_k)}\right|
\le
\frac{p_+}{\alpha}\left|\frac{w_k(x)}{\lambda}\right|^{p(x+h_k)}
=
\frac{p_+}{\alpha}F(\alpha,k)<\infty.
\]
Therefore, by the theorem on the differentiation under the sign of the
Lebesgue integral (see, e.g., \cite[Theorem~10.39]{A74}), the function
$F(\lambda,k)$ is differentiable in $\lambda\in(\alpha,\beta)$ and
\[
\frac{\partial F}{\partial\lambda}(\lambda,k)=
-\int_\R\frac{p(x+h_k)}{\lambda}\left|\frac{w_k(x)}{\lambda}\right|^{p(x+h_k)}dx.
\]
Since $[\alpha,\beta]$ was chosen arbitrarily, we conclude that $F(\lambda,k)$
is differentiable in the first variable on $\R_+$ and
\[
\frac{\partial F}{\partial\lambda}(\lambda,k)<0\quad\mbox{for}\quad\lambda\in\R_+.
\]
Thus, $F(\lambda,k)$ is strictly increasing and continuous in $\lambda\in\R_+$.
From this observation, \eqref{eq:key-2}, and \eqref{eq:key-17} we obtain in view of
Lemma~\ref{le:implicit} that there exist a number $k_2\in\N$ and a unique
sequence $\{\lambda(k)\}_{k=k_2}^\infty$ such that $F(\lambda(k),k)=1$ for all
$k\ge k_2$ and
%%%
\begin{equation}\label{eq:key-18}
\lim_{k\to\infty}\lambda(k)=\|w\|_q.
\end{equation}
%%%
On the other hand, taking into account that $F(\lambda,k)$ is strictly
decreasing and continuous, we see that
%%%
\begin{equation}\label{eq:key-19}
\|V_{h_k}w_k\|_{p(\cdot)}=\inf\big\{\lambda>0:\ F(\lambda,k)\le 1\big\}=\lambda(k).
\end{equation}
%%%
Combining \eqref{eq:key-18} and \eqref{eq:key-19}, we arrive at \eqref{eq:key-1}.
\qed
\end{proof}
%%%%%%%%%%%%%%%%%%%%%%%%%%%%%%%%%%%%%%%%%%%%%%%%%%%%%%%%%%%%%%%%%%%%%%%%%%%
\section{Proof of the main result}\label{sec:main-proof}
\subsection{Verification of the hypotheses of the key lemma}
We start with the following consequence of the Kronecker theorem on
almost periodic functions (see, e.g., \cite[Theorem~1.12]{BKS02}).
%%%%%%%%%%%%%%%%%%%%%%%%%%%%%%%%%%%%%%%%%%%%%%%%%%%%%%%%%%%%%%%%%%%%%%%%%%
\begin{lemma}[{see \cite[Lemma~10.2]{BKS02}}]
\label{le:Kronecker}
If $a_1,\dots,a_M\in AP_{N\times N}^0$ is a finite collection of almost
periodic polynomials, then there exists a sequence $\{h_m\}_{m=1}^\infty\subset\R$
such that $h_m\to+\infty$ {\rm(}resp. $h_m\to-\infty${\rm)} as $m\to\infty$ and
\[
\lim_{m\to\infty}\|a_j(\cdot+h_m)-a_j(\cdot)\|_{L_{N\times N}^\infty(\R)}=0
\]
for all $j\in\{1,\dots,M\}$.
\end{lemma}
%%%%%%%%%%%%%%%%%%%%%%%%%%%%%%%%%%%%%%%%%%%%%%%%%%%%%%%%%%%%%%%%%%%%%%%%%%%
The operator $S$ behaves extremely well on smooth compactly supported functions.
More precisely, we have the following.
%%%%%%%%%%%%%%%%%%%%%%%%%%%%%%%%%%%%%%%%%%%%%%%%%%%%%%%%%%%%%%%%%%%%%%%%%%
\begin{lemma}\label{le:Privalov}
If $\varphi\in C_c^\infty(\R)$, then $S\varphi\in C(\R)$ and there is a constant
$C_\varphi>0$ such that
\[
|(S\varphi)(x)|\le\frac{C_\varphi}{1+|x|}\quad(x\in\R).
\]
\end{lemma}
%%%%%%%%%%%%%%%%%%%%%%%%%%%%%%%%%%%%%%%%%%%%%%%%%%%%%%%%%%%%%%%%%%%%%%%%%%
\begin{proof}
The continuity of $S\varphi$ is a consequence of the Privalov theorem
(see, e.g., \cite[Chap.~II, Section~6.9]{S70}). For the pointwise estimate
for $S\varphi$, see e.g. \cite[Exercise~4.1.2(a)]{G08}.
\qed
\end{proof}
%%%%%%%%%%%%%%%%%%%%%%%%%%%%%%%%%%%%%%%%%%%%%%%%%%%%%%%%%%%%%%%%%%%%%%%%%%
Assume that $\alpha,\beta\in\{1,\dots,N\}$ and let $a_{\alpha\beta}$ denote the
$(\alpha,\beta)$-entry of a matrix function $a\in L_{N\times N}^\infty(\R)$.
%%%%%%%%%%%%%%%%%%%%%%%%%%%%%%%%%%%%%%%%%%%%%%%%%%%%%%%%%%%%%%%%%%%%%%%%%%%
\begin{lemma}\label{le:verification}
Let $\varphi\in C_c^\infty(\R)$. Suppose $a_l,a_r\in AP_{N\times N}^0$,
$a_0\in [C_0]_{N\times N}$, and
\[
a=(1-u)a_l+ua_r+a_0.
\]
Then
\begin{enumerate}[(a)]
\item
there exists a sequence $\{h_m\}_{m=1}^\infty$ such that
$h_m\to+\infty$ as $m\to\infty$ and  $w$, $\{w_m\}_{m=1}^\infty$
given by
%%%
\begin{equation}\label{eq:verification-1}
w:=\big((a_r)_{\alpha\beta}P+Q\big)\varphi,
\quad
w_m:=V_{-h_m}(a_{\alpha\beta}P+Q)V_{h_m}\varphi
\end{equation}
%%%
or
%%%
\begin{equation}\label{eq:verification-2}
w:=\big(\overline{(a_r)_{\alpha\beta}}P+Q\big)\varphi,
\quad
w_m:=V_{-h_m}(\overline{a_{\alpha\beta}}P+Q)V_{h_m}\varphi,
\end{equation}
%%%
where $\alpha,\beta\in\{1,\dots,N\}$, satisfy hypotheses {\rm(i)}
and {\rm(ii)} of Lemma~{\rm\ref{le:key}};

\item
there exists a sequence $\{h_m\}_{m=1}^\infty$ such that
$h_m\to-\infty$ as $m\to\infty$ and $w$, $\{w_m\}_{m=1}^\infty$
given by
\[
w:=\big((a_l)_{\alpha\beta}P+Q\big)\varphi,
\quad
w_m:=V_{-h_m}(a_{\alpha\beta}P+Q)V_{h_m}\varphi
\]
or
\[
w:=\big(\overline{(a_l)_{\beta\alpha}}P+Q\big)\varphi,
\quad
w_m:=V_{-h_m}(\overline{a_{\beta\alpha}}P+Q)V_{h_m}\varphi,
\]
where $\alpha,\beta\in\{1,\dots,N\}$, satisfy hypotheses {\rm(i)}
and {\rm(ii)} of Lemma~{\rm\ref{le:key}}.
\end{enumerate}
\end{lemma}
%%%%%%%%%%%%%%%%%%%%%%%%%%%%%%%%%%%%%%%%%%%%%%%%%%%%%%%%%%%%%%%%%%%%%%%%%%%
\begin{proof}
(a) By Lemma~\ref{le:Kronecker}, there exists a sequence $\{h_m\}_{m=1}^\infty$
such that
%%%
\begin{equation}\label{eq:verification-3}
\lim_{m\to\infty}\|a_r(\cdot+h_m)-a_r(\cdot)\|_{L_{N\times N}^\infty(\R)}=0.
\end{equation}
%%%
Fix $\alpha,\beta\in\{1,\dots,N\}$ and consider the pair given in \eqref{eq:verification-1}.
It is easy to see that for $m\in\N$ and $x\in\R$,
%%%
\begin{equation}\label{eq:verification-4}
w_m(x)
=(V_{-h_m}a_{\alpha\beta}V_{h_m}P\varphi)(x)+(Q\varphi)(x)
=
a_{\alpha\beta}(x+h_m)(P\varphi)(x)+(Q\varphi)(x).
\end{equation}
%%%
From Lemma~\ref{le:Privalov} it follows that $P\varphi,Q\varphi\in C(\R)$
and there exists a constant $C_\varphi>0$ such that
%%%
\begin{equation}\label{eq:verification-5}
|(P\varphi)(x)|\le\frac{\widetilde{C}_\varphi}{1+|x|},
\quad
|(Q\varphi)(x)|\le\frac{\widetilde{C}_\varphi}{1+|x|},
\end{equation}
%%%
where $\widetilde{C}_\varphi:=(C_\varphi+\|\varphi\|_\infty)/2$. From
\eqref{eq:verification-4}--\eqref{eq:verification-5} it follows that for
$m\in\N$ and $x\in\R$,
\[
|w_m(x)|\le\frac{\|a_{\alpha\beta}\|_\infty \widetilde{C}_\varphi}{1+|x|},
\quad
|w(x)|\le\frac{\|(a_r)_{\alpha\beta}\|_\infty\widetilde{C}_\varphi}{1+|x|}.
\]
These inequalities mean that hypothesis (ii) of Lemma~\ref{le:key}
holds for $w$, $w_m$ given by
\eqref{eq:verification-1} with $\alpha,\beta\in\{1,\dots,N\}$.

From \eqref{eq:verification-4} and the representation
\[
a=(1-u)(a_l-a_r)+a_0+a_r
\]
we obtain for every $m\in\N$ and every $x\in\R$,
%%%
\begin{align}
|w_m(x)- w(x)|
= &
|a_{\alpha\beta}(x+h_m)-(a_r)_{\alpha\beta}(x)|\,|(P\varphi)(x)|
\nonumber
\\
\le &
\big(|1-u(x+h_m)|+|(a_0)_{\alpha\beta}(x+h_m)|
+|(a_r)_{\alpha\beta}(x+h_m)-(a_r)_{\alpha\beta}(x)|\big)|(P\varphi)(x)|.
\label{eq:verification-6}
\end{align}
%%%
Let $J\subset\R$ be a closed segment. Since $1-u(+\infty)=0$ and
$(a_0)_{\alpha\beta}(+\infty)=0$, we have
%%%
\begin{equation}\label{eq:verification-7}
\lim_{m\to\infty}\sup_{x\in J}|1-u(x+h_m)|=0,
\quad
\lim_{k\to\infty}\sup_{x\in J}|(a_0)_{\alpha\beta}(x+h_m)|=0.
\end{equation}
%%%
From \eqref{eq:verification-3} we also have
%%%
\begin{equation}\label{eq:verification-8}
\lim_{m\to\infty}\sup_{x\in J}|(a_r)_{\alpha\beta}(x+h_m)-(a_r)_{\alpha\beta}(x)|=0.
\end{equation}
%%%
The first inequality in \eqref{eq:verification-5} yields
%%%
\begin{equation}\label{eq:verification-9}
\sup_{x\in J}|(P\varphi)(x)|\le\widetilde{C}_\varphi\sup_{x\in J}\frac{1}{1+|x|}<\infty.
\end{equation}
%%%
From \eqref{eq:verification-6}--\eqref{eq:verification-9} we deduce that
\[
\lim_{m\to\infty}\sup_{x\in J}|w_m(x)-w(x)|=0,
\]
which finishes the verification of hypothesis (i) of Lemma~\ref{le:key}
for $w$, $w_m$ given by
\eqref{eq:verification-1} with $\alpha,\beta\in\{1,\dots,N\}$. The proof
for $w$, $w_m$ given by \eqref{eq:verification-2} is similar.
Part (a) is proved. The proof of part (b) is analogous.
\qed
\end{proof}
%%%%%%%%%%%%%%%%%%%%%%%%%%%%%%%%%%%%%%%%%%%%%%%%%%%%%%%%%%%%%%%%%%%%%%%%%%%
\subsection{Proof of Theorem~\ref{th:main}}
(a) The idea of the proof is borrowed from \cite[Theorem~6.5]{RS08}.
Since the operator $aP+Q$ is Fredholm on $L_N^{p(\cdot)}(\R)$,
its adjoint operator $(aP+Q)^*$ is Fredholm on the dual space in view of
Theorem~\ref{th:Fredholm-duality}. From Corollary~\ref{co:adjoints-sio} and
Lemma~\ref{le:relations-sio} it follows that
\[
(aP+Q)^*=Pa^*I+Q=A_1(a^*P+Q)A_2,
\]
where the operators $A_1:=I+Pa^*Q$ and $A_2:=I-Qa^*P$ are
invertible on $L_N^{p'(\cdot)}(\R)$. From this equality and
Lemma~\ref{le:Atkinson} we deduce that the operator $a^*P+Q$ is
Fredholm on $L_N^{p'(\cdot)}(\R)$. Therefore, due to
Theorem~\ref{th:regularization}, the operator $A:=aP+Q$ admits a
left regularizer on $L_N^{p(\cdot)}(\R)$ and the operator
$A':=a^*P+Q$ admits a left regularizer on $L_N^{p'(\cdot)}(\R)$.
That is, there exist operators $B\in\cB(L_N^{p(\cdot)}(\R))$,
$K\in\cK(L_N^{p(\cdot)}(\R))$ and $B'\in\cB(L_N^{p'(\cdot)}(\R))$,
$K'\in\cK(L_N^{p'(\cdot)}(\R))$ such that
%%%
\begin{equation}\label{eq:main-1}
BA-K=I,
\quad
B'A'-K'=I.
\end{equation}
%%%
Since $a\in SAP_{N\times N}$, there exist $a_l,a_r\in AP_{N\times N}$
and $a_0\in [C_0]_{N\times N}$ such that
%%%
\begin{equation}\label{eq:main-2}
a=(1-u)a_l+ua_r+a_0.
\end{equation}
%%%
By the definition of $AP$, there exist sequences
$\{a_l^{(j)}\}_{j=1}^\infty,\{a_r^{(j)}\}_{j=1}^\infty\subset AP_{N\times N}^0$
such that
%%%
\begin{equation}\label{eq:main-3}
\lim_{j\to\infty}\|a_l^{(j)}-a_l\|_{L_{N\times N}^\infty(\R)}=0,
\quad
\lim_{j\to\infty}\|a_r^{(j)}-a_r\|_{L_{N\times N}^\infty(\R)}=0.
\end{equation}
%%%
Let $a_j:=(1-u)a_l^{(j)}+ua_r^{(j)}+a_0$ and
\[
A_j:=a_jP+Q,
\quad
A_j':=a_j^*P+Q,
\quad
R_j:=a_r^{(j)}P+Q,
\quad
R_j':=(a_r^{(j)})^*P+Q.
\]
Put
\[
J:=\left[\liminf_{x\to+\infty}p(x),\limsup_{x\to+\infty}p(x)\right],
\quad
J':=\left[\liminf_{x\to+\infty}p'(x),\limsup_{x\to+\infty}p'(x)\right].
\]
It is well known that the norm of the operator $S$ on the standard
Lebesgue spaces is calculated by
\[
\|S\|_{\cB(L^q(\R))}=\left\{\begin{array}{lll}
\displaystyle
\tan\frac{\pi}{2q} &\mbox{if}& 1<q\le 2,
\\[3mm]
\displaystyle
\cot\frac{\pi}{2q} &\mbox{if}& 2\le q< \infty
\end{array}\right.
\]
(see, e.g., \cite[Chap.~13, Theorem~1.3]{GK92}). Hence
\[
\sup_{q\in J\cup J'}\max\left\{\|P\|_{L_N^q(\R))},\|Q\|_{L_N^q(\R))}\right\}=:M<\infty.
\]
If we denote $R:=a_rP+Q$ and $R':=P+a_r^*Q$, then
%%%
\begin{align}
\label{eq:main-4}
&
\sup_{q\in J}\|R-R_j\|_{\cB(L_N^q(\R))}
\le
C_NM\|a_r-a_r^{(j)}\|_{L_{N\times N}^\infty(\R)},
\\
\label{eq:main-5}
&
\sup_{q'\in J'}\|R'-R_j'\|_{\cB(L_N^{q'}(\R))}
\le
C_NM\|a_r-a_r^{(j)}\|_{L_{N\times N}^\infty(\R)},
\end{align}
%%%
where the constant $C_N>0$ depends only on $N$.
From \eqref{eq:main-2}--\eqref{eq:main-3} it follows that
%%%
\begin{align}
\label{eq:main-6}
&
\|A-A_j\|_{\cB(L_N^{p(\cdot)}(\R))}<\frac{1}{2\|B\|_{\cB(L_N^{p(\cdot)}(\R))}},
\\
\label{eq:main-7}
&
\|A'-A_j'\|_{\cB(L_N^{p'(\cdot)}(\R))}<\frac{1}{2\|B'\|_{\cB(L_N^{p'(\cdot)}(\R))}}
\end{align}
%%%
for sufficiently large $j$. Further, from
\eqref{eq:main-3}--\eqref{eq:main-5} we also deduce that
%%%
\begin{align}
\label{eq:main-8}
&
\sup_{q\in J}\|R-R_j\|_{\cB(L_N^q(\R))}
<\frac{1}{2\|B\|_{\cB(L_N^{p(\cdot)}(\R))}},
\\
\label{eq:main-9}
&
\sup_{q'\in J'}\|R'-R_j'\|_{\cB(L_N^{q'}(\R))}
<\frac{1}{2\|B'\|_{\cB(L_N^{p'(\cdot)}(\R))}}
\end{align}
%%%
for sufficiently large $j$. Fix $j$ such that all inequalities
\eqref{eq:main-6}--\eqref{eq:main-9} are fulfilled simultaneously.

From the first equality in \eqref{eq:main-1} and \eqref{eq:main-6} it follows
that for every $f\in L_N^{p(\cdot)}(\R)$,
\[
\begin{split}
\|f\|_{L_N^{p(\cdot)}(\R)}
&\le
\|B\|_{\cB(L_N^{p(\cdot)}(\R))}\|Af\|_{L_N^{p(\cdot)}(\R)}+\|Kf\|_{L_N^{p(\cdot)}(\R)}
\\
&\le
\|B\|_{\cB(L_N^{p(\cdot)}(\R))}
\big(\|A_jf\|_{L_N^{p(\cdot)}(\R)}+\|Af-A_jf\|_{L_N^{p(\cdot)}(\R)}\big)+
\|Kf\|_{L_N^{p(\cdot)}(\R)}
\\
&\le
\|B\|_{\cB(L_N^{p(\cdot)}(\R))}\|A_jf\|_{L_N^{p(\cdot)}(\R)}+\frac{1}{2}\|f\|_{L_N^{p(\cdot)}(\R)}+\|Kf\|_{L_N^{p(\cdot)}(\R)}.
\end{split}
\]
Hence for all $f\in L_N^{p(\cdot)}(\R)$,
%%%
\begin{equation}\label{eq:main-10}
\|f\|_{L_N^{p(\cdot)}(\R)}\le 2\|B\|_{\cB(L_N^{p(\cdot)}(\R))}\|A_jf\|_{L_N^{p(\cdot)}(\R)}+2\|Kf\|_{L_N^{p(\cdot)}(\R)}.
\end{equation}
%%%
Analogously, from the second equality in \eqref{eq:main-1}
and \eqref{eq:main-7} we obtain for $g\in L_N^{p'(\cdot)}(\R)$,
%%%
\begin{equation}\label{eq:main-11}
\|g\|_{L_N^{p'(\cdot)}(\R)}
\le
2\|B'\|_{\cB(L_N^{p'(\cdot)}(\R))}\|A_j'g\|_{L_N^{p'(\cdot)}(\R)}+
2\|K'g\|_{L_N^{p'(\cdot)}(\R)}.
\end{equation}
%%%

Let $\psi_n$ be as in Lemma~\ref{le:convergence-compact-multiplication}.
It is clear that $\Psi_n:=\operatorname{diag}\{\psi_nI,\dots,\psi_nI\}$
is an idempotent, that is, $\Psi_n^2=\Psi_n$. By Lemma~\ref{le:convergence-compact-multiplication}(b),
there exists an $n\in\N$ such that
\[
\|K\Psi_n\|_{\cB(L_N^{p(\cdot)}(\R))}\le \frac{1}{4},
\quad
\|K'\Psi_n\|_{\cB(L_N^{p'(\cdot)}(\R))}\le \frac{1}{4}.
\]
Hence for all $f\in L_N^{p(\cdot)}(\R)$,
%%%
\begin{equation}\label{eq:main-12}
\|K\Psi_nf\|_{L_N^{p(\cdot)}(\R)}
=
\|K\Psi_n^2f\|_{L_N^{p(\cdot)}(\R)}
\le
\|K\Psi_n\|_{\cB(L_N^{p(\cdot)}(\R))}\|\Psi_nf\|_{L_N^{p(\cdot)}(\R)}
\le\frac{1}{4}\|\Psi_nf\|_{L_N^{p(\cdot)}(\R)},
\end{equation}
%%%
and similarly
%%%
\begin{equation}\label{eq:main-13}
\|K'\Psi_ng\|_{L_N^{p'(\cdot)}(\R)}
\le
\frac{1}{4}\|\Psi_ng\|_{L_N^{p'(\cdot)}(\R)}.
\end{equation}
%%%
From \eqref{eq:main-10} and \eqref{eq:main-12} it follows that for all $f\in L_N^{p(\cdot)}(\R)$,
%%%
\begin{equation}\label{eq:main-14}
\|\Psi_nf\|_{L_N^{p(\cdot)}(\R)}\le 4\|B\|_{\cB(L_N^{p(\cdot)}(\R))}\|A_j\Psi_n f\|_{L_N^{p(\cdot)}(\R)}
\end{equation}
%%%
In the same way, from \eqref{eq:main-11} and \eqref{eq:main-13}
we obtain for all $g\in L_N^{p'(\cdot)}(\R)$,
%%%
\begin{equation}\label{eq:main-15}
\|\Psi_n g\|_{L_N^{p'(\cdot)}(\R)}\le 4\|B'\|_{\cB(L_N^{p'(\cdot)}(\R))}\|A_j'\Psi_n g\|_{L_N^{p'(\cdot)}(\R)}.
\end{equation}
%%%

Let $\varphi\in [C_c^\infty(\R)]_N$. In view of
Lemma~\ref{le:verification}(a), there exists a sequence
$\{h_m\}_{m=1}^\infty$ such that $h_m\to+\infty$ as $m\to\infty$ and
each of the functions given by
\[
w_{\alpha\beta}:=\big((a_r^{(j)})_{\alpha\beta}P+Q\big)\varphi_\beta,
\quad
(w_{\alpha\beta})_m:=V_{-h_m}((a_j)_{\alpha\beta}P+Q)V_{h_m}\varphi_\beta
\]
and
\[
w_{\alpha\beta}':=\big(\overline{(a_r^{(j)})_{\beta\alpha}}P+Q\big)\varphi_\beta,
\quad
(w_{\alpha\beta})_m':=V_{-h_m}(\overline{(a_j)_{\beta\alpha}}P+Q)V_{h_m}\varphi_\beta
\]
%%%
for $\alpha,\beta\in\{1,\dots,N\}$ satisfies hypotheses (i) and (ii)
of Lemma~\ref{le:key}.

For $h\in\R$, let the translation operator $V_h$ be defined on $L_N^{p(\cdot)}(\R)$
and on $L_N^{p'(\cdot)}(\R)$ elementwise (although it may be unbounded on these spaces).
It is easy to see that there exists an $m_0\in\N$ such that
%%%
\begin{equation}\label{eq:main-16}
\Psi_nV_{h_m}\varphi=V_{h_m}\varphi \quad\text{for all}\quad m\ge m_0.
\end{equation}
%%%
Then from \eqref{eq:main-14} and \eqref{eq:main-16} it follows that
for all $\varphi\in [C_c^\infty(\R)]_N$ and all $m\ge m_0$,
%%%
\begin{align}
\|V_{h_m}\varphi\|_{L_N^{p(\cdot)}(\R)}
&=
\|\Psi_nV_{h_m}\varphi\|_{L_N^{p(\cdot)}(\R)}
\nonumber
\\
&\le
4\|B\|_{\cB(L_N^{p(\cdot)}(\R))}\|A_j\Psi_nV_{h_m}\varphi\|_{L_N^{p(\cdot)}(\R)}
\nonumber
\\
&=
4\|B\|_{\cB(L_N^{p(\cdot)}(\R))}\|A_jV_{h_m}\varphi\|_{L_N^{p(\cdot)}(\R)}
\nonumber
\\
&=
4\|B\|_{\cB(L_N^{p(\cdot)}(\R))}\|V_{h_m}(V_{-h_m}A_jV_{h_m}\varphi)\|_{L_N^{p(\cdot)}(\R)}.
\label{eq:main-17}
\end{align}
%%%
Analogously, from \eqref{eq:main-15} and \eqref{eq:main-16} we get
for all $\varphi\in [C_c^\infty(\R)]_N$ and all $m\ge m_0$,
%%%
\begin{equation}\label{eq:main-18}
\|V_{h_m}\varphi\|_{L_N^{p'(\cdot)}(\R)}
\le
4\|B'\|_{\cB(L_N^{p'(\cdot)}(\R))}\|V_{h_m}(V_{-h_m}A_j'V_{h_m}\varphi)\|_{L_N^{p'(\cdot)}(\R)}.
\end{equation}
%%%

Since the sequence $\{p(h_m)\}_{m=1}^\infty$ is bounded,
$p_-\le p(h_m)\le p_+$ for all $m\in\N$, there exists its convergent
subsequence $\{ p(h_{m_k}) \}_{k=1}^\infty$. Let
\[
q_r:=\lim_{k\to\infty}p(h_{m_k}).
\]
It is clear that $q_r\in J$. Taking into account \eqref{eq:exponents} we also see
that
\[
\lim_{k\to\infty}p'(h_{m_k})=q_r/(q_r-1)=:q_r'\in J'.
\]
Applying Lemma~\ref{le:key} to
\[
w_{\alpha\beta}:=\big((a_r^{(j)})_{\alpha\beta}P+Q\big)\varphi_\beta,
\quad
(w_{\alpha\beta})_{m_k}:=V_{-h_{m_k}}\big((a_j)_{\alpha\beta}P+Q\big)V_{h_{m_k}}\varphi_\beta
\]
with $\alpha,\beta\in\{1,\dots,N\}$, we obtain
\[
\lim_{k\to\infty}
\|V_{h_{m_k}}(V_{-h_{m_k}}(A_j)_{\alpha\beta}V_{h_{m_k}}\varphi_\beta\|_{p(\cdot)}
=
\lim_{k\to\infty}\|V_{h_{m_k}}(w_{\alpha\beta})_{m_k}\|_{p(\cdot)}
=
\|w_{\alpha,\beta}\|_{q_r}=\|(R_j)_{\alpha\beta}\varphi_\beta\|_{q_r}.
\]
Then
%%%
\begin{equation}\label{eq:main-19}
\lim_{k\to\infty}\|V_{h_{m_k}}(V_{-h_{m_k}}A_jV_{h_{m_k}}\varphi)\|_{L_N^{p(\cdot)}(\R)}
=
\|R_j\varphi\|_{L_N^{q_r}(\R)}.
\end{equation}
%%%
Analogously, applying Lemma~\ref{le:key} to
\[
w_{\alpha\beta}':=\big(\overline{(a_r^{(j)})_{\beta\alpha}}P+Q\big)\varphi_\beta,
\quad
(w_{\alpha\beta})_{m_k}':=V_{-h_{m_k}}(\overline{(a_j)_{\beta\alpha}}P+Q)V_{h_{m_k}}\varphi_\beta
\]
with $\alpha,\beta\in\{1,\dots,N\}$ on the dual space, we get
%%%
\begin{equation}\label{eq:main-20}
\lim_{k\to\infty}\|V_{h_{m_k}}(V_{-h_{m_k}}A_j'V_{h_{m_k}}\varphi)\|_{L_N^{p'(\cdot)}(\R)}
=
\|R_j'\varphi\|_{L_N^{q_r'}(\R)}.
\end{equation}
%%%
Finally, applying Lemma~\ref{le:key} to the constant sequences $w_k=\varphi_\beta$ and $w=\varphi_\beta$
for all $\beta\in\{1,\dots,N\}$, we get
%%%
\begin{equation}\label{eq:main-21}
\lim_{k\to\infty}\|V_{h_{m_k}}\varphi\|_{L_N^{p(\cdot)}(\R)}=\|\varphi\|_{L_N^{q_r}(\R)},
\quad
\lim_{k\to\infty}\|V_{h_{m_k}}\varphi\|_{L_N^{p'(\cdot)}(\R)}=\|\varphi\|_{L_N^{q_r'}(\R)}.
\end{equation}
%%%
Inequalities \eqref{eq:main-17} and \eqref{eq:main-18}, in particular, imply that
for all $k\in\N$ and $\varphi\in [C_c^\infty(\R)]_N$,
%%%
\begin{align*}
&
\|V_{h_{m_k}}\varphi\|_{L_N^{p(\cdot)}(\R)}
\le
4\|B\|_{\cB(L_N^{p(\cdot)}(\R))}\|V_{h_{m_k}}(V_{-h_{m_k}}A_jV_{h_{m_k}}\varphi)\|_{L_N^{p(\cdot)}(\R)},
\\
&
\|V_{h_{m_k}}\varphi\|_{L_N^{p'(\cdot)}(\R)}
\le
4\|B'\|_{\cB(L_N^{p'(\cdot)}(\R))}\|V_{h_{m_k}}(V_{-h_{m_k}}A_j'V_{h_{m_k}}\varphi)\|_{L_N^{p'(\cdot)}(\R)}.
\end{align*}
%%%
Passing in these inequalities to the limit as $k\to\infty$ and taking into
account equalities \eqref{eq:main-19}--\eqref{eq:main-21}, we obtain
for all $\varphi\in [C_c^\infty(\R)]_N$,
%%%
\begin{align}
\label{eq:main-22}
&
\|\varphi\|_{L_N^{q_r}(\R)}
\le
4\|B\|_{\cB(L_N^{p(\cdot)}(\R))}
\|R_j\varphi\|_{L_N^{q_r}(\R)},
\\
\label{eq:main-23}
&
\|\varphi\|_{L_N^{q_r'}(\R)}
\le
4\|B'\|_{\cB(L_N^{p'(\cdot)}(\R))}
\|R_j'\varphi\|_{L_N^{q_r'}(\R)}.
\end{align}
%%%
From inequalities \eqref{eq:main-8} and \eqref{eq:main-22} we obtain
\[
\begin{split}
\|\varphi\|_{L_N^{q_r}(\R)}
\le &
4\|B\|_{\cB(L_N^{p(\cdot)}(\R))}\|R\varphi\|_{L_N^{q_r}(\R)}
+
4\|B\|_{\cB(L_N^{p(\cdot)}(\R))}
\|R-R_j\|_{\cB(L_N^{q_r}(\R))}\|\varphi\|_{L_N^{q_r}(\R)}
\\
\le &
4\|B\|_{\cB(L_N^{p(\cdot)}(\R))}\|R\varphi\|_{L_N^{q_r}(\R)}
+\frac{1}{2}\|\varphi\|_{L_N^{q_r}(\R)}.
\end{split}
\]
Hence, for all $\varphi\in [C_c^\infty(\R)]_N$,
%%%
\begin{equation}\label{eq:main-24}
\|\varphi\|_{L_N^{q_r}(\R)}
\le
8\|B\|_{\cB(L_N^{p(\cdot)}(\R))}\|R\varphi\|_{L_N^{q_r}(\R)}.
\end{equation}
%%%
Let $f\in L_N^{q_r}(\R)$ and $\{\varphi_k\}_{k=1}^\infty\subset [C_c^\infty(\R)]_N$
be a sequence such that
\[
\lim_{k\to\infty}\|f-\varphi_k\|_{L_N^{q_r}(\R)}=0.
\]
From this equality and \eqref{eq:main-24} it follows that
\[
\|f\|_{L_N^{q_r}(\R)}
=
\lim_{k\to\infty}\|\varphi_k\|_{L_N^{q_r}(\R)}
\le
8\|B\|_{\cB(L_N^{p(\cdot)}(\R))}\lim_{k\to\infty}\|R\varphi_k\|_{L_N^{q_r}(\R)}
\\
=
8\|B\|_{\cB(L_N^{p(\cdot)}(\R))}\|Rf\|_{L_N^{q_r}(\R)}.
\]
Therefore
%%%
\begin{equation}\label{eq:main-25}
0<\frac{1}{8\|B\|_{\cB(L_N^{p(\cdot)}(\R))}}\le \cJ(R;L_N^{q_r}(\R)).
\end{equation}
%%%
Arguing analogously and starting with \eqref{eq:main-9} and \eqref{eq:main-23},
we obtain
%%%
\begin{equation}\label{eq:main-26}
0<\frac{1}{8\|B'\|_{\cB(L_N^{p'(\cdot)}(\R))}}\le \cJ(R';L_N^{q_r'}(\R)).
\end{equation}

From Corollary~\ref{co:adjoints-sio} and Lemma~\ref{le:relations-sio} we obtain
\[
R^*=(a_rP+Q)^*=Pa_r^*I+Q=A_3R'A_4,
\]
where $A_3:=I+Pa_r^*Q$ and $A_4:=I-Qa_r^*P$ are invertible on
$L_N^{q_r'}(\R)$. From this equality,
Lemma~\ref{le:moduli-supermult} and Theorem~\ref{th:Kurbatov} it
follows that
%%%
\begin{align}
\cJ(R^*; L_N^{q_r'}(\R))
%&
\ge
\cJ(A_3;L_N^{q_r'}(\R))\cdot
\cJ(R';L_N^{q_r'}(\R))\cdot
\cJ(A_4;L_N^{q_r'}(\R))
=
\frac{\cJ(R';L_N^{q_r'}(\R))}
{\|A_3^{-1}\|_{\cB(L_N^{q_r'}(\R))}\|A_4^{-1}\|_{\cB(L_N^{q_r'}(\R))}}.
\label{eq:main-27}
\end{align}
%%%
From \eqref{eq:relations-sio-1} we see that
%%%
\begin{align}
\|A_3^{-1}\|_{\cB(L_N^{q_r'}(\R))}
&=
\|I-Pa_r^*Q\|_{\cB(L_N^{q_r'}(\R))}
\nonumber
\\
&\le
1+\|P\|_{\cB(L_N^{q_r'}(\R))}\|a_r^*I\|_{\cB(L_N^{q_r'}(\R))}\|Q\|_{\cB(L_N^{q_r'}(\R))}
\nonumber
\\
&\le
1+C_N\|a_r\|_{L_{N\times N}^\infty(\R)}M^2
\label{eq:main-28}
\end{align}
%%%
and analogously
%%%
\begin{equation}\label{eq:main-29}
\|A_4^{-1}\|_{\cB(L_N^{q_r'}(\R))}
\le
1+C_N\|a_r\|_{L_{N\times N}^\infty(\R)}M^2.
\end{equation}
%%%
Combining \eqref{eq:main-26}--\eqref{eq:main-29}, we arrive at
\[
\cJ(R^*; L_N^{q_r'}(\R))
\ge
\frac{\cJ(R';L_N^{q_r'}(\R))}
{\big(1+C_N\|a_r\|_{L_{N\times N}^\infty(\R)}M^2\big)^2}
\ge
\frac{\big(8\|B'\|_{\cB(L_N^{p'(\cdot)}(\R))}\big)^{-1}}
{\big(1+C_N\|a_r\|_{L_{N\times N}^\infty(\R)}M^2\big)^2}
=:
M_1>0.
\]
From this inequality and Lemma~\ref{le:injection-surjection-duality} we
conclude that
%%%
\begin{equation}\label{eq:main-30}
\cQ(R; L_N^{q_r}(\R))\ge M_1>0.
\end{equation}
%%%
Finally, inequalities \eqref{eq:main-25}, \eqref{eq:main-30} and
Theorem~\ref{th:Kurbatov} imply that the operator $R=a_rP+Q$
is invertible on the standard Lebesgue space $L_N^{q_r}(\R)$.
\qed
%%%%%%%%%%%%%%%%%%%%%%%%%%%%%%%%%%%%%%%%%%%%%%%%%%%%%%%%%%%%%%%%%%%%%%%%%%%

\textbf{Acknowledgement.}
The authors would like to thank the anonymous referee for pointing
out an inaccuracy in the first version of the paper.

%%%%%%%%%%%%%%%%%%%%%%%%%%%%%%%%%%%%%%%%%%%%%%%%%%%%%%%%%%%%%%%%%%%%%%%%%%%

\end{document}